\documentclass[11pt]{article}

\usepackage{amsthm}
\usepackage[nocompress]{cite}
\usepackage{graphicx}
\usepackage{subcaption}
\usepackage{hyperref}
\hypersetup{
    colorlinks=true,
    linkcolor=blue,
    anchorcolor=blue,
    citecolor=blue,
}
\usepackage{enumitem}
\usepackage{amssymb, amsfonts, amsmath, nicefrac}
\usepackage{color}
\usepackage{mathtools} 
\usepackage{cleveref} 
\usepackage{todonotes} 
\usepackage{mathrsfs} 
\usepackage[normalem]{ulem}

\crefdefaultlabelformat{#2\textup{#1}#3}
\Crefname{assumption}{Assumption}{Assumptions}
\crefname{assumption}{Assumption}{Assumptions}
\crefname{theorem}{Theorem}{Theorems}
\Crefname{theorem}{Theorem}{Theorems}
\crefname{lemma}{Lemma}{Lemmas}
\Crefname{lemma}{Lemma}{Lemmas}
\crefname{proposition}{Proposition}{Propositions}
\Crefname{proposition}{Proposition}{Propositions}   
\Crefrangeformat{assumption}{Assumptions #3\textup{#1}#4--#5\textup{#2}#6}

\newtheorem{theorem}{Theorem}[section]
\newtheorem{lemma}{Lemma}[section]
\newtheorem{proposition}{Proposition}[section]
\newtheorem{example}{Example}[section]
\newtheorem{remark}{Remark}[section]

\newtheorem{assumption}{Assumption}[section]

\newcommand\tsum{\textstyle\sum\nolimits}
\newcommand{\sU}{\mathscr{U}}
\newcommand{\cV}{{\mathfrak V}}
\newcommand{\sigmaprop}{\sigma_{t, \mathrm{prop}}}
\newcommand{\sigmacurr}{\sigma_{t, \mathrm{curr}}}

\def\argmin{\mathop{\rm arg\,min}}
\newcommand{\D}{{\cal D}}
\newcommand{\U}{{\cal U}}
\newcommand{\F}{{\cal F}}

\newcommand{\C}{{\cal C}}
\newcommand{\X}{{\cal X}}
\newcommand{\Y}{{\cal Y}}

\newcommand{\bx}{\boldsymbol{x}}
\newcommand{\w}{\omega}
\DeclareMathOperator{\cov}{Cov}
\DeclareMathOperator{\var}{Var}
\newcommand{\bbr}{\mathbb{R}} 
\newcommand{\bbe}{\mathbb{E}} 
\newcommand{\dst}{\rightsquigarrow\,}

\numberwithin{equation}{section}

\begin{document}

\title{\bf Central Limit Theorems for Sample Average Approximations in Stochastic Optimal Control\thanks{\textbf{Funding}:
The research of the first author was supported by the National Science Foundation under Grant No.\ DMS-2410944.
The research of the second author was partially supported by Air Force Office of Scientific Research (AFOSR)
under Grant FA9550-22-1-0244.
}}
\author{{\bf Johannes Milz}\thanks{H.\ Milton Stewart School of Industrial and Systems Engineering, Georgia Institute of Technology, Atlanta, Georgia
30332, USA (\texttt{johannes.milz@isye.gatech.edu}, \texttt{ashapiro@isye.gatech.edu})}
\and  {\bf Alexander Shapiro}\footnotemark[2]
}

\maketitle

\begin{abstract}
We establish central limit theorems for the Sample Average
Approximation (SAA) method in discrete-time, finite-horizon stochastic
optimal control. Our analysis is based on an abstract limit theorem for
stochastic backward recursions, which yields a recursive
characterization of the limiting laws. Applied to the dynamic programming principle,
this framework gives Gaussian limits for SAA value functions under
unique optimal policies. The asymptotic variance at each stage
decomposes into a current-stage variance and a propagated future
variance, demonstrating how statistical uncertainty accumulates backward
through time. We also apply the framework to the linear quadratic
regulator, derive explicit limiting laws and variance formulas,
and provide numerical illustrations of the resulting variance
decomposition. Finally, we discuss the form of the limit laws under
nonunique optimal policies.
\end{abstract}


\section{Introduction}
\label{sec-intr}

An approach to solving  stochastic programming problems is to approximate  the ``true'' distribution of the corresponding  random   vector by  the  empirical distribution  based on a randomly generated sample. This approach became known as the Sample Average Approximation (SAA) method \cite{Kleywegt2002}. For multistage (linear) stochastic programming problems this approach was used in 
\cite{STCS2013} for a numerical study of planning   the Brazilian
interconnected power system.
For static (one-stage) stochastic programs   statistical inference of the SAA method is well developed (cf., \cite[Section 5.1]{SDR}). It is known that under mild regularity conditions the optimal value of the SAA problem asymptotically has a normal distribution, provided that the true problem has a unique optimal solution. On the other hand, little is known about asymptotics of the SAA method applied to multistage stochastic optimization problems. The main goal of this paper is to derive Central Limit Theorem (CLT)-type asymptotics for Stochastic Optimal Control  (SOC) in discrete
time.
One motivation for
studying such CLTs is the construction of confidence intervals to certify
termination of optimization methods such as \texttt{MSPPy} \cite{ding2019}.

\paragraph{Contributions}
The primary contributions of this paper are as follows:
\begin{enumerate}[nosep]
\item We develop an abstract statistical limit theorem for stochastic
backward recursions. This result provides a recursive description of the
limiting laws arising from stochastic recursions, and
serves as the main tool for our subsequent analysis of SAA in SOC\@. This result requires a
stochastic
equicontinuity-type condition.

\item We apply this abstract framework to discrete-time, finite-horizon
SOC and establish CLTs for SAA
value functions. The limiting Gaussian process at each stage admits a
recursive representation, and its asymptotic variance decomposes into a
\emph{current-stage variance} and a \emph{propagated future variance}.
This decomposition makes precise how sampling uncertainty is transmitted
backward through time.

\item We provide sufficient conditions for the stochastic
equicontinuity-type condition needed in the CLT analysis. 
One of these conditions is verified for an
inventory control problem.

\item We apply the abstract  framework to the linear
quadratic regulator (LQR) and derive explicit formulas for the limiting
Gaussian recursion and the associated variance decomposition. This
provides a transparent analytical and numerical illustration of the
current-stage and propagated variance components.

\item We discuss the form of the limit laws in the presence of
nonunique optimal policies. In that case, the directional derivatives of
the dynamic programming operators become nonlinear, and the resulting limit laws are
generally not Gaussian.
\end{enumerate}

\paragraph{Related work}

The sample complexity of multistage stochastic programming has been established in \cite{shapnem:05,Shapiro2006}. However, as   noted above, the literature on the limit distributions of SAA value functions and SAA  solutions and policies remains relatively sparse. For instance, \cite{Eichhorn2007} analyzes limit
distributions in the context of two-stage integer programming, but the results
do not appear to extend to multistage settings.
Limit theorems for infinite-horizon, discounted stochastic optimal control are discussed  in \cite{cltshapcheng21}.
More recently, \cite{Zhang2025} analyzes the asymptotic behavior of data-driven policies in a periodic review stochastic inventory control problem.

As we highlight, the dynamic programming recursions and their SAA
approximation underlying our analysis admit an interpretation within
the framework of Z-estimation. We refer the reader to
\cite{Vaart2023,Kosorok2008} for comprehensive treatments of this
theory.

\paragraph{Outline}
\Cref{sec:backward-recursion} develops an abstract limit theorem for
stochastic backward recursions. \Cref{sec-basic} introduces the
stochastic optimal control model and assumptions. \Cref{sec-cltsaa}
applies the framework to SAA Bellman recursions, establishes CLTs for
SAA value functions and optimal values, and derives the decomposition
into current-stage and propagated future variance, together with
sufficient conditions for stochastic equicontinuity. \Cref{sec:LQR}
applies the framework to LQR, and
\Cref{sec:lqr-numerical} presents numerical illustrations. \Cref{sec-summary}
provides concluding discussions. The appendix contains some technical proofs, and a
derivation of limit laws under nonunique optimal policies in \Cref{app-B}.

\paragraph{Notation}
We use the following notation.
We define $ [a]_+ \coloneqq \max\{0,a\}$.
For a process $\xi_1,\ldots$, we denote by $\xi_{[t]}=(\xi_1,\ldots,\xi_{t})$ its history. By `$\coloneqq$' we mean ``equal by definition".
 By `$\dst$' we denote convergence in distribution. By $\delta_\xi$ we denote the measure of mass one at $\xi$ (Dirac measure).
We denote by $\mathcal{N}(\mu, \Sigma)$ the normal distribution
with mean vector $\mu$ and covariance matrix $\Sigma$.
For a compact metric space $S$ we denote by
$C(S)$ the space of continuous functions $\phi:S\to \bbr$
  equipped   with the
supremum norm $\|\phi\|_\infty=\sup_{x\in S}|\phi(x)|$.
More generally, 
the space of continuous mappings between a  a compact metric space $S$
and a Banach space $\mathcal{Y}$ is denoted by $C(S, \mathcal{Y})$
and equipped with the standard supremum norm.
For two random variables $X$ and $Y$, we denote by $\bbe[X|Y]$ or $\bbe_{|Y}[X]$ the conditional expectation of $X$ given $Y$. 

Let $\mathcal{X}$ and $\mathcal{Y}$ be Banach spaces,
$\mathcal{D}\subset \mathcal{X}$, $\bar x\in \mathcal{D}$.
We say that a mapping  $\Phi\colon \mathcal{D}\to \mathcal{Y}$ is
Hadamard directionally differentiable at $\bar x$
tangentially to $\mathcal{D}$ if for all
sequences $\tau_k\in \bbr$ and $h_k\in \X$ such that 
 $\tau_k\downarrow 0$, $h_k\to h$ and  
$\bar x+\tau_k h_k\in \mathcal{D}$, it follows that the following limit exists 
(in  the norm topology of  $\mathcal{Y}$):
\begin{equation}
\label{derivhadamard}
\Phi'(\bar x; h) \coloneqq 
\lim_{k \to \infty}\frac{\Phi(\bar x+\tau_k h_k)-\Phi(\bar x)}{\tau_k}.
\end{equation}
Note that the conditions $\tau_k\downarrow 0$, $h_k\to h$ and  
$\bar x+\tau_k h_k\in \mathcal{D}$ mean that $h$ is an element of the so-called contingent cone $T_\D(\bar{x})$ to $\D$ at $\bar{x}$. We have that the cone $T_\D(\bar{x})\subset \X$ is closed and it follows from Hadamard directional differentiability that the directional derivative $\Phi'(\bar x; \cdot):T_\D(\bar{x})\to \Y$ is positively homogeneous and continuous. 
Moreover, if the limit  \eqref{derivhadamard} holds,  $T_\D(\bar{x})$ is a linear subspace of $\X$, and $\Phi'(\bar x; \cdot)$ is linear and bounded, 
then it is said that  $\Phi$ is
Hadamard   differentiable at $\bar x$
(tangentially to $\mathcal{D}$).
In our applications, $\mathcal{D}$ are affine subspaces, in which case
the contingent cone $T_\mathcal{D}(\bar x)$ is a linear subspace of $\mathcal{X}$.
Throughout the paper, we use the fact that if $\Phi$ is Lipschitz continuous and
directionally differentiable at $\bar x$, then $\Phi$ is Hadamard
directionally differentiable at $\bar x$ and $\Phi'(\bar x; \cdot)$ is Lipschitz continuous
(e.g.,\cite[Proposition 2.49]{BS2000}).

For a compact subset \(\mathcal{Y}_0\) of a metric space \(\mathcal{Y}\),
the $\varepsilon$-metric entropy 
\(\mathcal{H}(\varepsilon, \mathcal{Y}_0)\)
with $\varepsilon  >0$
is the natural logarithm of the \(\varepsilon\)-covering number of
\(\mathcal{Y}_0\)
(with respect to $\mathcal{Y}$).

\section{Statistical limit theorems for backward recursions}
\label{sec:backward-recursion}

We formulate an abstract backward recursion that is  applied to
the dynamic programming equations in later sections. For a   horizon 
  \(T \in \mathbb{N}\) and 
$t=1,\ldots,T+1$, let \(\mathcal{C}_t\) be a separable
Banach space with the respective norm $\|\cdot\|_{\mathcal{C}_t}$,  and let \(\mathcal{D}_t \subset \mathcal{C}_t\). 
Recall that a random element $X$ in $\mathcal{C}_t$ is a measurable mapping $X:\Omega\to \C_t$ from a probability space $(\Omega,\F,P)$ into the space $\C_t$ equipped with its Borel sigma-algebra.
A mapping
\(Y:\Omega\times\mathcal{D}_{t+1}\to\mathcal{C}_t\)
is called a jointly measurable random mapping if it is measurable with
respect to the product sigma-algebra on
\(\Omega\times\mathcal{D}_{t+1}\).
For every \(x\in\mathcal{D}_{t+1}\), the map
\(\omega\mapsto Y(\omega,x)\) is then a random element in
\(\mathcal{C}_t\).
To simplify notation, we suppress the variable \(\omega\) and write
\(Y:\mathcal{D}_{t+1}\to\mathcal{C}_t\), understanding that \(Y(x)\)
denotes the \(\mathcal{C}_t\)-valued random element
\(\omega\mapsto Y(\omega,x)\).

An important example is the space $\C_t\coloneqq C(S)$ of continuous functions defined on compact metric space $S$.   In that case the respective  random element can be viewed as a random function $X(\w,s)$.  It is said that the random element $X\in C(S)$ is Gaussian if for any $s_1,\ldots,s_m\in S$ the (finite dimensional)   distribution  of the random vector  $(X(\w, s_1),\ldots,X(\w,s_m))$ is Gaussian. 

For $t=1,\ldots,T$,  we consider a  sequence of mappings
$
\mathcal{B}_t \colon \mathcal{D}_{t+1} \to \mathcal{C}_t
$,
and let
\begin{equation}\label{random}
\hat{\mathcal{B}}_{t,N}:\mathcal{D}_{t+1}\to \mathcal{C}_t
\end{equation}
be a sequence of jointly measurable random mappings indexed by
\(N\in\mathbb{N}\). The mappings \(\hat{\mathcal{B}}_{t,N}\) are viewed
as stochastic approximations of \(\mathcal{B}_t\).
For a  (fixed)  terminal element \(W_{T+1} \in \mathcal{D}_{T+1}\), we consider 
the associated deterministic backward recursion   defined by
\[
W_t \coloneqq \mathcal{B}_t(W_{t+1}),
\qquad t=T,\ldots,1,
\]
and its random counterpart   defined by $\hat W_{T+1,N} \coloneqq W_{T+1}$ and
\[
\hat W_{t,N} \coloneqq \hat{\mathcal{B}}_{t,N}(\hat W_{t+1,N}),
\qquad t=T,\ldots,1.
\]

We derive, by backward induction, statistical limit theorems for \(N^{1/2}\big(\hat W_{t,N}-W_t\big)\).
We use the following  technical conditions allowing us to establish such limit laws. Recall that it is assumed that $\C_t$ are separable Banach spaces.

\begin{assumption}[{Statistical limit conditions for  backward recursions}]
~
\label{assume:general}
For  $t \in \{1, \ldots, T\}$ and $N \in \mathbb{N}$:
\begin{enumerate}[nosep,label=\textup{(\roman*)}]
\item The mapping $\mathcal{B}_t \colon \mathcal{D}_{t+1} \to \mathcal{C}_t$ 
is Borel measurable, and Hadamard directionally differentiable
at $W_{t+1}$ tangentially to $\mathcal{D}_{t+1}$,
and $\mathcal{B}_t\mathcal{D}_{t+1} \subset \mathcal{D}_t$.
\item 
The random mapping 
$\hat{\mathcal{B}}_{t,N} \colon \mathcal{D}_{t+1} \to \C_t$
is jointly measurable, independent of the random
element $\hat{W}_{t+1,N}$, and 
$\hat{\mathcal{B}}_{t,N}\mathcal{D}_{t+1} \subset \mathcal{D}_t$
with probability one.
\item There exists  a random  element $\mathfrak{H}_t$ 
 with values in $\mathcal{C}_t$
such that 
\begin{align}
\label{eq:general-intermediate-limit}
N^{1/2}
\big[
\hat{\mathcal{B}}_{t,N}- \mathcal{B}_t
\big]
(W_{t+1})
\dst \mathfrak{H}_t.
\end{align}

\end{enumerate}
\end{assumption}

\Cref{assume:general} implies
that with probability one, the random element $\hat{W}_{t+1,N}$ takes values in $\mathcal{D}_{t+1}$.
 \Cref{assume:general} implies (induction start) that for $W_{T+1} \in \mathcal{D}_{T+1}$ and $\mathfrak{G}_T \coloneqq \mathfrak{H}_T$:
\begin{align*}
N^{1/2}
\big[
\hat{W}_{T,N}- W_{T}
\big]
=
N^{1/2}
\big[
\hat{\mathcal{B}}_{T,N}- \mathcal{B}_{T}
\big]
(W_{T+1})
\dst \mathfrak{G}_T.
\end{align*}

Recall the definition \eqref{derivhadamard} of Hadamard differentiability.

\begin{theorem}[{Statistical limit theorems for backward recursions}]
\label{thm:general-clt}
Let $t \in \{1, \ldots, T\}$, and let \Cref{assume:general} hold.
Suppose that  $N^{1/2}(\hat{W}_{t+1,N} - W_{t+1}) \dst \mathfrak{G}_{t+1}$
in $\mathcal{C}_{t+1}$ \emph{(induction hypothesis)}, and that 
\begin{align}
\label{eq:general-continuity}
\big\| [\hat{\mathcal{B}}_{t,N}-
\mathcal{B}_{t} ](\hat{W}_{t+1,N})-
 [
\hat{\mathcal{B}}_{t,N}- \mathcal{B}_t
 ]
(W_{t+1})
\big \|_{\mathcal{C}_t}
= o_p(N^{-1/2}).
\end{align}
Then $N^{1/2}(\hat{W}_{t,N} - W_{t})$ converges in distribution  to a random element 
$\mathfrak{G}_{t}$ with values in $\mathcal{C}_t$, and the following recursion holds  
\begin{align}
\label{eq:general-law-recursion}
\mathfrak{G}_{t}=
\mathcal{B}_t^\prime(W_{t+1}; \mathfrak{G}_{t+1})
+
\mathfrak{H}_{t}.
\end{align}
Moreover, if   $\mathcal{B}_t$ is Hadamard differentiable at $W_{t+1}$, 
and $\mathfrak{G}_{t+1}$ and $\mathfrak{H}_t$ are Gaussian,
then $\mathfrak{G}_{t}$ is Gaussian.
\end{theorem}

\begin{proof}
We consider the error decomposition
\begin{align}
\label{eq:general-error-decompostion}
\begin{aligned}
N^{1/2}(\hat{W}_{t,N} - W_{t})
& = \Big[N^{1/2}\big[
\hat{\mathcal{B}}_{t,N}
-
\mathcal{B}_{t}
\big](\hat{W}_{t+1,N})
-
N^{1/2}
\big[
\hat{\mathcal{B}}_{t,N}- \mathcal{B}_t
\big]
(W_{t+1})
\Big]
\\
& \quad +
N^{1/2}
\big[
\hat{\mathcal{B}}_{t,N}- \mathcal{B}_t
\big]
(W_{t+1})
+
N^{1/2}(\mathcal{B}_t \hat{W}_{t+1,N} - \mathcal{B}_t W_{t+1}).
\end{aligned}
\end{align}
The first term on the right-hand side of \eqref{eq:general-error-decompostion} 
converges to $0$ in probability as $N \to \infty$. 
The second and third terms are independent. 
Combined with \eqref{eq:general-intermediate-limit}, 
the Delta Theorem
(see, e.g., Theorem~9.74 in \cite{SDR}), 
the induction hypothesis, and Slutsky's theorem, we obtain the first two
assertions. 

Under the additional hypotheses, \eqref{eq:general-law-recursion} ensures
that $\mathfrak{G}_t$ is Gaussian.
\end{proof}%

We comment on the additional assumption \eqref{eq:general-continuity}.

\begin{remark}[{Stochastic equicontinuity-type condition}]
\normalfont
In the statement of \Cref{thm:general-clt}
we introduced the  new assumption  \eqref{eq:general-continuity}, which plays a key
role in the asymptotic analysis.
This condition involves the stochastic equicontinuity of the family of
operators
$
N^{1/2} \big[\hat{\mathcal{B}}_{t,N} - \mathcal{B}_t\big]
$
at $W_{t+1}$. For this reason, we refer to \eqref{eq:general-continuity}
as a \emph{stochastic equicontinuity-type condition}
(of
$
N^{1/2} \big[\hat{\mathcal{B}}_{t,N} - \mathcal{B}_t\big]
$
at $W_{t+1}$).
The circumstances under which this condition holds can be subtle,
as they depend on the interaction between empirical approximation,
operator structure, and the regularity of $W_{t+1}$.
We discuss simple sufficient conditions in \Cref{rem:Bellman-Lipschitz-1} below.
\hfill $\square$
\end{remark}

The backward recursion can be interpreted as a Z-estimation problem as 
discussed next.

\begin{remark}[{Relationship to Z-estimation}]
\normalfont 
Defining $\Psi_t(w_t, w_{t+1}) \coloneqq w_t  - \mathcal{B}_t w_{t+1}$
and its random counterpart
$\hat{\Psi}_{t,N}(w_t, w_{t+1}) \coloneqq w_t  - \hat{\mathcal{B}}_{t,N} w_{t+1}$,
the backward recursions become Z-estimation problems.
The stochastic equicontinuity-type condition \eqref{eq:general-continuity} 
is related to that in eq.~(3.3.2) in \cite{Vaart2023}.
$\hfill \square$
\end{remark}

We provide conditions sufficient for the
stochastic equicontinuity-type
condition \eqref{eq:general-continuity}.

\begin{remark}[{Sufficient conditions for the stochastic equicontinuity-type
condition}]
\label{rem:Bellman-Lipschitz-1}
\normalfont
We introduce two sufficient conditions for the stochastic
equicontinuity-type condition \eqref{eq:general-continuity} to hold.
\begin{enumerate}[nosep,label=\textup{(\roman*)}]
\item 
Since
$N^{1/2}(\hat{W}_{t+1,N} - W_{t+1}) \dst \mathfrak{G}_{t+1}$
implies that
$N^{1/2}\|\hat{W}_{t+1,N}-W_{t+1}\|_{_{\mathcal{C}_{t+1}}} = O_p(1)$,
a sufficient condition for \eqref{eq:general-continuity} to hold is
\begin{align}
\label{difference}
\big\|\big[\hat{\mathcal{B}}_{t,N}
-
\mathcal{B}_{t}\big](\hat{W}_{t+1,N})
-
\big[
\hat{\mathcal{B}}_{t,N}- \mathcal{B}_t
\big]
(W_{t+1})
\big\|_{_{\mathcal{C}_{t}}}
=
o_p(1)\|\hat{W}_{t+1,N}-W_{t+1}\|_{_{\mathcal{C}_{t+1}}}.
\end{align}
Condition \eqref{difference} imposes a Lipschitz-type bound on the difference of the
operators when applied to $\hat{W}_{t+1,N}$ and $W_{t+1}$, with a
Lipschitz constant that vanishes in probability.

\item Suppose that $\hat{\mathcal{B}}_{t,N}$ and $\mathcal{B}_t$ are
continuous, $\D_{t+1}$ is compact,
$\|\hat{W}_{t+1,N}-W_{t+1}\|_{_{\mathcal{C}_{t+1}}} = o_p(1)$, and
$
N^{1/2}
\big[
\hat{\mathcal{B}}_{t,N}- \mathcal{B}_t
\big]
$
converges in distribution to a random element in
$C(\mathcal{D}_{t+1}, \mathcal{C}_t)$.
Then \eqref{eq:general-continuity} holds true.
While this implication is essentially known
(cf.\ \cite[pp.~52--53]{Pollard1990}),
let us provide a concise verification.
We define
$
Z_N \coloneqq
N^{1/2}\big[\hat{\mathcal{B}}_{t,N}- \mathcal{B}_t\big]
$.
By assumption,
$Z_N \dst Z$
for some random element
$Z$ with values in $C(\mathcal{D}_{t+1}, \mathcal{C}_t)$.
Now define
$
\Upsilon \colon
C(\mathcal{D}_{t+1}, \mathcal{C}_t)\times \mathcal{D}_{t+1}
\to \mathcal{C}_t
$
by 
$\Upsilon(\upsilon,x)\coloneqq \upsilon(x)-\upsilon(W_{t+1}).$
Since $\Upsilon$ is continuous, the continuous mapping theorem 
and independence give
$
\Upsilon(Z_N,\hat{W}_{t+1,N})
\dst
\Upsilon(Z,W_{t+1}) = 0
$,
which yields \eqref{eq:general-continuity}.
\end{enumerate}
$\hfill \square$
\end{remark}

\Cref{thm:general-clt} can be generalized to allow for dependent
operators $\hat{\mathcal{B}}_{t,N}$.

\begin{remark}
\normalfont 
If  $\hat{\mathcal{B}}_{t,N}$ and $\hat{W}_{t+1, N}$
are dependent, then the assertion of \Cref{thm:general-clt} 
remains valid if $N^{1/2}(\hat{W}_{t+1,N} - W_{t+1}) \dst \mathfrak{G}_{t+1}$
in $\mathcal{C}_{t+1}$ 
and \eqref{eq:general-intermediate-limit} are replaced by the joint convergence
in distribution
\begin{align*}
N^{1/2}\big(
 [
\hat{\mathcal{B}}_{t,N}- \mathcal{B}_t]
(W_{t+1}),
\hat{W}_{t+1,N} - W_{t+1}
\big)
\dst
\big(\mathfrak{H}_t, 
\mathfrak{G}_{t+1} ).
\end{align*}
In this case, we obtain the convergence in distribution of the sum of the second
and third terms in \eqref{eq:general-error-decompostion} via an application 
of the Delta Method to the
mapping $\mathcal{C}_t \times \mathcal{D}_{t+1} \ni (u,w) \mapsto
u + \mathcal{B}_t(w)$.
$\hfill \square$
\end{remark}

\section{Stochastic optimal control in discrete time}
\label{sec-basic}

This section formulates the SOC problems considered throughout the
remainder of the manuscript.
We consider the    discrete time, finite horizon
SOC
 model (see, e.g., \cite{Bertsekas2005}):
\begin{equation}\label{soc}
\min\limits_{\pi\in \Pi}  \bbe^\pi\left [ \tsum_{t=1}^{T}
f_t(x_t,u_t,\xi_t)+f_{T+1}(x_{T+1})
\right],
\end{equation}
where $\Pi$ is the set of policies satisfying   the constraints
\begin{equation}\label{soc-b}
\Pi=\Big\{\pi=(\pi_1,\ldots,\pi_T):
u_t=\pi_t(x_t,\xi_{[t-1]}),
u_t\in \U_t, x_{t+1}=F_t(x_t,u_t,\xi_t),\;\;t=1,\ldots,T\Big\}.
\end{equation}
Here variables  $x_t\in \bbr^{n_t}$, $t=1,\ldots,T+1$, represent the state  of the system,   $u_t\in \bbr^{m_t}$,  $t=1,\ldots,T$, are controls,   $\xi_t\sim P_t$  are random vectors whose probability distribution $P_t$ is supported on a closed    subset $\Xi_t$ of $\bbr^{d_t}$, $f_t:\bbr^{n_t}\times\bbr^{m_t}\times\bbr^{d_t}\to \bbr$, $t=1,\ldots,T$, are cost functions,
$f_{T+1} \colon \bbr^{n_{T+1}} \to \bbr$
is the final cost function,
  $F_t:\bbr^{n_t}\times\bbr^{m_t}\times\bbr^{d_t}\to \bbr^{n_{t+1}}$ are (measurable) mappings, and $\U_t$  is a (nonempty)  subset of $\bbr^{m_t}$.
The values  $x_1$  and $\xi_0$ are  deterministic  (initial conditions); it is also possible to view $x_1$ as random with a given distribution;
however this is not essential for the following discussion.

\begin{assumption}
\label{ass-1}
{\rm (i)}
The probability distribution $P_t$ of $\xi_t$  does not depend on our decisions (on states and actions), $t=1,\ldots,T$.
{\rm (ii)}
The random process $\xi_1,\ldots,\xi_T$ is stagewise independent, i.e., the random vector  $\xi_{t+1}$ is independent of $\xi_{[t]}\coloneqq (\xi_1,\ldots,\xi_{t})$, $t=1,\ldots,T-1$.
\end{assumption}

The optimization in \eqref{soc}   is performed over policies  $\pi\in \Pi$  determined by  decisions  $u_t$ and state variables $x_t$ considered as  functions of $\xi_{[t-1]}=(\xi_1,\ldots,\xi_{t-1})$, $t=1,\ldots,T$,
and satisfying the feasibility constraints \eqref{soc-b}.
We also denote $\Xi_{[t]} \coloneqq \Xi_1 \times \cdots \times \Xi_t$.
  For the sake of simplicity, in order not to distract from the main message of the paper,  we assume that the control sets $\U_t$ do not depend on $x_t$.

\begin{remark}
\label{rem-pol}
{\rm
Under \Cref{ass-1}, and in particular by stagewise independence, it
suffices to consider policies of the form
$\{u_t=\pi_t(x_t)\}$, $t=1,\ldots,T$.

For a given policy $\pi\in \Pi$, the state variables in problem
\eqref{soc} are functions of $\xi_{[t-1]}$ and hence are random. On the
other hand, in some settings we use the same notation $x_t$ as a vector
in $\bbr^{n_t}$. To indicate when the state variables are viewed as
random, we use the bold face notation $\bx_t$.
}\hfill $\square$
\end{remark}

Suppose that at every stage $t=1,\ldots,T$, an
independent  and identically distributed
(iid) random sample $\xi_{ti}$, $i=1,2, \ldots$, of realizations\footnote{It is also possible to consider different sample sizes at different stages of the process. For the sake of simplicity we assume that the sample size $N$ is the same for every stage $t$. Usually  in numerical implementations the same sample size is used.} of the random vector $\xi_t\sim P_t$  is generated. 

\begin{assumption}
\label{ass-1'}
For each $t \in \{1,\ldots,T\}$, the random variables
$\xi_{t1}, \xi_{t2}, \ldots,$ are iid with common distribution $P_t$.
Moreover, the sample arrays
$
\{\xi_{ti}: i=1,2, \ldots\}
$,
$t=1,\ldots,T$,
are independent.
\end{assumption}

The SAA counterpart of problem \eqref{soc} is obtained by replacing the probability distributions $P_t$ with their empirical counterparts 
\begin{equation}\label{empirmes}
\hat{P}_{t,N}\coloneqq N^{-1}\tsum_{i=1}^N \delta_{\xi_{ti}}.
\end{equation}

The dynamic programming equations
for problem  \eqref{soc} are:   $V_{T+1}(x_{T+1})=f_{T+1}(x_{T+1})$, and
\begin{align}
\label{depen-1}
V_t(x_t) & =\inf_{u_t \in \U_t}
\bbe_{\xi_t\sim P_t}  \left [f_t(x_t,u_t,\xi_t)+
V_{t+1}\big(F_t(x_t,u_t,\xi_t) \big)\right],\;t=T,\ldots,1,
\end{align}
where the expectation is taken with respect to the probability distribution $P_t$ of $\xi_t$.
The SAA counterpart of  \eqref{depen-1} is:
$\hat{V}_{T+1,N}(x_{T+1})=f_{T+1}(x_{T+1})$, and 
\begin{align}
\label{depen-2}
\hat{V}_{t,N}(x_t) & =\inf_{u_t\in \U_t}\frac{1}{N}\sum_{i=1}^N
   \left [f_t(x_t,u_t,\xi_{ti})+
\hat{V}_{t+1,N}\big(F_t(x_t,u_t,\xi_{ti}) \big)\right],\;t=T,\ldots,1.
\end{align}
Note that under the considered assumptions, the value functions $V_t(\cdot)$ and 
$\hat{V}_{t,N}(\cdot)$ are continuous.

\section{Central limit theorems for nonparametric SAA value functions}
\label{sec-cltsaa}

This section develops CLTs for (nonparametric) SAA value functions
by applying the abstract  results of section \ref{sec:backward-recursion}, in particular  \Cref{thm:general-clt}. We first state the
assumptions used throughout the section and establish the terminal-stage CLT,
which serves as the base case for the backward induction.
  Our main result,
presented in \Cref{subsect:clt}, is an application of
\Cref{thm:general-clt} to the SAA dynamic programming recursion. It holds
under a stochastic equicontinuity-type condition comparing the SAA dynamic
programming operators with their true counterparts. 

\subsection{The terminal-stage CLT}

We begin by stating the assumptions used throughout this section and by
establishing the CLT at the terminal decision stage $t=T$. This result serves
as the base case for the backward induction developed in the next subsection.
At the terminal stage, the analysis becomes ``static'' and some known results can be applied in a straightforward way. 

\begin{assumption}
\label{ass-2}
For $t=1,\ldots,T$:
 {\rm   (i)} The sets $\X_t$, $\U_t$, and $\Xi_t$ are compact.
 {\rm   (ii)}  The functions $f_t:\X_t\times\U_t\times\Xi_t\to \bbr$
and $F_t:\X_t\times\U_t\times\Xi_t\to \X_{t+1}$, $t=1,\ldots,T$,   are continuous, and $f_{T+1} : \X_{T+1} \to \bbr$ is Lipschitz continuous.
 {\rm   (iii)}    There exists a nonnegative valued function $K_t(\xi_t)$
   such that $\bbe[K_t(\xi_t)^2]$ is finite, and
   for all $x_t,x'_t\in \X_t$,
   $u_t,u'_t\in \U_t$ and   $\xi_t\in \Xi_t$:
\begin{align*}
     |f_t(x_t,u_t,\xi_t)-f_t(x_t',u'_t,\xi_t)|& \le K_t(\xi_t)(\|x_t-x'_t\|+\|u_t-u'_t\|),\\
      \|F_t(x_t,u_t,\xi_t)-F_t(x_t',u'_t,\xi_t)\|&\le K_t(\xi_t)(\|x_t-x'_t\|+\|u_t-u'_t\|).
\end{align*}
   {\rm   (iv)}
   For every $x_t\in \X_t$, problem \eqref{depen-1} has
   a unique minimizer $u_t^*= \pi^*_t(x_t)$.
\end{assumption}

The compactness assumptions in \Cref{ass-2}(i) simplify the analysis in two
ways. First, the compactness of $\mathcal{U}_t$ ensures directional
differentiability of the infimum operator associated with the Bellman
recursion. Second, the compactness of $\mathcal{X}_t$, $\mathcal{U}_t$, and
$\Xi_t$, together with \Cref{ass-2}(iii), allows us to verify
\eqref{eq:general-intermediate-limit} using standard functional CLTs for the
associated empirical Lipschitz processes. Note that \Cref{ass-2}(i)--(iii) implies that the value functions $V_t(\cdot)$ and $\hat{V}_{t,N}(\cdot)$ are continuous. This follows from the continuity of the involved functions and compactness arguments (e.g., \cite[Proposition 4.4]{BS2000}).

Let us start by verifying \eqref{eq:general-intermediate-limit} at the last stage $t=T$.
We consider
   \begin{align}
   \label{PhiT}
    \Phi_T(x_T,u_T,\xi_T) & \coloneqq  f_T(x_T,u_T,\xi_T)+
f_{T+1} (F_T(x_T,u_T,\xi_T)  )\\
 \hat{\Phi}_{T,N}(x_T,u_T)
 & \coloneqq   N^{-1}\tsum_{i=1}^N \Phi_T(x_T,u_T,\xi_{Ti})=\bbe_{\xi_T \sim \hat{P}_{T,N}}[\Phi_{T}(x_T,u_T,\xi_T)],
 \end{align}
 where $\hat{P}_{T,N}$ is the empirical counterpart of $P_T$ (see
\eqref{empirmes}).

The next result verifies \eqref{eq:general-intermediate-limit} for $t=T$,
and computes the covariance function of the limiting law.

\begin{proposition}
\label{pr-1}
Suppose that for $t=T$,   
  \Cref{ass-2}  holds. Then 
$N^{1/2} (\hat{V}_{T,N} -V_{T}  )$ converges in distribution to a
Gaussian random element $\mathfrak{G}_T$ with values in $C(\X_T)$ with zero mean and covariance function 
\begin{equation}\label{covphi-a} \Gamma_T(x_T,x'_T)\coloneqq  \cov(\mathfrak{G}_T(x_T),\mathfrak{G}_T(x'_T)),
\end{equation}
 given by
\begin{equation}\label{covphi-1}
\Gamma_T(x_T,x'_T)=\cov\left(\Phi_T(x_T,\pi^*_T(x_T),\xi_T),
\Phi_T(x'_T,\pi^*_T(x'_T),\xi_T)\right).
\end{equation}
\end{proposition}

Our proof of \Cref{pr-1} is based on the following lemma, which we use
repeatedly throughout the manuscript. 

\begin{lemma}
\label{lem:Gt-Hadamard}
Suppose that \Cref{ass-2}\textup{(i)--(iii)}  holds, 
and let
  $t \in \{1, \ldots, T\}$.
 We consider the mapping $G_t \colon C(\mathcal{X}_t \times \mathcal{U}_t) \to C(\mathcal{X}_t)$ 
defined  by $[G_t(\phi)](x_t) \coloneqq \inf_{u_t \in \mathcal{U}_t} \phi(x_t, u_t)$.
Then $G_t$  
is Lipschitz continuous, and  directionally differentiable with 
   \begin{equation}\label{dirderiv}
 [ G_t'(\phi;\eta)](x_t)=\inf_{u_t\in \sU_t^*(\phi;x_t)}\eta(x_t,u_t), 
\quad \eta\in C(\X_t\times \U_t),
 \end{equation}
 where $\sU_t^*(\phi;x_t)\coloneqq\argmin_{u_t\in \U_t}\phi(x_t,u_t)$. 
\end{lemma}

\begin{proof}
Lipschitz continuity of $G_t$ is a standard result. 
The directional differentiability
follows from an application of Danskin's theorem  
(see, e.g., \cite[Theorem 9.26]{SDR}).
\end{proof}

As it was already pointed out, Lipschitz continuity  and  directional  differentiability of $G_t$ 
imply its Hadamard directional  differentiability. 

\begin{proof}[{Proof of \Cref{pr-1}}]
The derivations are   similar to
  \cite{sha-91} (see also \cite[Theorem 5.7]{SDR}); we briefly outline the derivations.  
    By the functional CLT, we have under \Cref{ass-2} that $N^{1/2} (\hat{\Phi}_{T,N} -\bbe[\Phi_T] )$  converges in distribution to  a Gaussian random element in
$C(\X_T\times \U_T)$ with zero mean and  covariance function
$
\cov\left(\Phi_T(x_T,u_T,\xi_T),
\Phi_T(x'_T,u'_T,\xi_T)\right)
$; see, e.g., \cite[Example 19.7]{van2000}.

Now, we consider the mapping $G_T$ defined in \Cref{lem:Gt-Hadamard}. 
By \Cref{ass-2}(iv), the solution set  $\sU_T^*(\bbe[\Phi_T(\cdot,\cdot,\xi_T)];x_T)$ is a singleton $\{\pi^*_T(x_T)\}$  for every $x_T\in \X_T$. The asymptotics then follow by the Delta Theorem  (see \cite[Theorem 9.74]{SDR}).
\end{proof}

In particular it follows that for a given $x_T\in \X_T$, $N^{1/2} \big(\hat{V}_{T,N} (x_T)-V_{T} (x_T) \big)$ converges in distribution to a mean zero Gaussian random variable
with variance
$\var[\Phi_T(x_T,\pi^*_T(x_T),\xi_T)]$.

\begin{remark}
\label{rem-nonunique}
{\rm
If the set $\sU_T^*(\phi;x_T)$ in \eqref{dirderiv} is not a singleton, we still can use \eqref{dirderiv} to conclude (cf., \cite[Theorem 5.7]{SDR}) that for any $x_T\in \X_T$, 
\begin{equation}\label{asymvalue}
  \hat{V}_{T,N}(x_T)=\inf_{u_T\in \U_T^*(x_T)}\hat{\Phi}_{T,N}(x_T,u_T)+o_p(N^{-1/2}),
\end{equation}
where
\begin{equation}\label{asymval-2}
 \U_T^*(x_T)\coloneqq\argmin_{u_T\in \U_T} 
\bbe_{\xi_T \sim P_T}  \left [f_T(x_T,u_T,\xi_T)+
f_{T+1}\big(F_T(x_T,u_T,\xi_T) \big)\right].
\end{equation}
Moreover, 
\eqref{asymvalue} holds uniformly in $x_T\in \X_T$. 
If the set $\U_T^*(x_T)$ is not a singleton, i.e., \Cref{ass-2}\textup{(iv)} does not hold, then by \eqref{asymvalue} we have that 
   $N^{1/2} \big(\hat{V}_{T,N} (x_T)-V_{T} (x_T) \big)$ may not be asymptotically normal. 
This highlights the relevance of \Cref{ass-2}(iv), about uniqueness of the minimizer $u_t^*= \pi^*_t(x_t)$, in obtaining the Gaussian limits. We will discuss this further, for general time periods, in Appendix \ref{app-B}. 
} $\hfill \square$
\end{remark}

\subsection{Central limit theorems for SAA value functions}
\label{subsect:clt}

Here, we develop the inductive step of the CLT for the SAA value functions. We present a recursive theorem that, under a stochastic equicontinuity-type assumption, constructs the limiting distribution at stage $t$ from the one at stage $t+1$. The resulting limit process is shown to be a sum of two independent Gaussian components: one representing the propagated uncertainty from future stages and another capturing the sampling error from the current stage.

We begin by defining the
dynamic programming operators and their empirical counterparts.
Let \Cref{ass-2} hold.
For \(t=1,\ldots,T\), we define the dynamic programming operator
\(\mathcal{T}_t \colon C(\mathcal{X}_{t+1}) \to  C(\mathcal{X}_{t})\)
by
\begin{equation}\label{saaderiv-1}
(\mathcal{T}_t V)(x_t) \coloneqq  \inf_{u_t\in \mathcal{U}_t}\, \mathbb{E}_{\xi_t\sim P_t}\bigl[f_t(x_t,u_t,\xi_t) + V\bigl(F_t(x_t,u_t,\xi_t)\bigr)\bigr],
\end{equation}
and  the SAA dynamic programming operator \(\hat{ \mathcal{T}}_{t,N} \colon C(\mathcal{X}_{t+1}) \to  C(\mathcal{X}_{t})\) by
\begin{equation}\label{saaderiv-2}
(\hat{\mathcal{T}}_{t,N} V)(x_t) \coloneqq  \inf_{u_t\in \mathcal{U}_t}\, \frac{1} {N} \sum_{i=1}^N \Bigl[f_t(x_t,u_t,\xi_{ti}) + V\bigl(F_t(x_t,u_t,\xi_{ti})\bigr)\Bigr].
\end{equation}
These operators are well-defined owing to \Cref{ass-2} and
\cite[Proposition~4.4]{BS2000}. Note that
$\hat{V}_{t,N}=\hat{\mathcal{T}}_{t,N}\hat{V}_{t+1,N}$.
We define
\begin{equation}\label{eq-vt}
 \hat{\mathcal{V}}_{t,N}\coloneqq  \hat{\mathcal{T}}_{t,N}V_{t+1},
\end{equation}
and 
$\Phi_t \colon \mathcal{X}_t \times \mathcal{U}_t \times \Xi_t \to \bbr$
and
$\Psi_{t+1} \colon \mathcal{X}_t \times \mathcal{U}_t \times
C(\mathcal{X}_{t+1}) \to \bbr$ by
\begin{align}
\label{Phi-t}
\Phi_t(x_t, u_t, \xi_t)
& \coloneqq
f_t(x_t,u_t,\xi_{t}) + V_{t+1}\bigl(F_t(x_t,u_t,\xi_{t})\bigr),
\\
\label{Psi-t}
\Psi_{t+1}(x_{t}, u_t, W_{t+1})
&\coloneqq
\mathbb{E}_{\xi_t\sim P_t}[W_{t+1}(F_t(x_t, u_t ,\xi_t))].
\end{align}
Of course for $t=T$, the above $\Phi_T$ coincides with the one defined in \eqref{PhiT}.

The following theorem establishes CLTs for SAA value functions through
induction backward in time. Notably, the CLT for $t=T$, the base case,
is provided by \Cref{pr-1}, which ensures that 
$N^{1/2}(\hat{V}_{T,N} - V_{T}) \dst \mathfrak{G}_{T}$,
where $\mathfrak{G}_{T}$ is a mean-zero Gaussian random
element in $C(\mathcal{X}_T)$. 
Recall that $\hat{\mathcal{T}}_{t,N}$ and $\hat{\mathcal{V}}_{t,N}$ are defined in
\eqref{saaderiv-2} and \eqref{eq-vt}, while $\Phi_t$ and $\Psi_{t+1}$ are
defined in \eqref{Phi-t} and \eqref{Psi-t}, respectively.
\begin{theorem}
\label{thm:clt}
Suppose that 
 \Cref{ass-1,ass-1',ass-2} hold. 
Let $t \in \{1, \ldots, T\}$.
Suppose that $N^{1/2}(\hat{V}_{t+1,N} - V_{t+1}) \dst \mathfrak{G}_{t+1}$,
where $\mathfrak{G}_{t+1}$ is a mean-zero Gaussian process in $C(\mathcal{X}_{t+1})$ \emph{(induction hypothesis)}. Then
\begin{align}
\label{eq:clt-mathcalV}
N^{1/2}
(\hat{\mathcal{V}}_{t,N} - V_t)
\dst
\mathfrak{H}_t,
\end{align}
where $\mathfrak{H}_t$ is a mean-zero
Gaussian process in $C(\mathcal{X}_t)$
with covariance function
\begin{align}
\label{covHt}
(x_t, x_t^\prime)
\mapsto
\cov(
\Phi_t(x_t, \pi^*_t(x_t), \xi_t),
\Phi_t(x_t^\prime,  \pi^*_t(x_t^\prime), \xi_t)
).
\end{align}

Suppose further that
\begin{align}
\label{eq:continuity}
\big\|\big[\hat{\mathcal{T}}_{t,N}
-
\mathcal{T}_{t}\big](\hat{V}_{t+1,N})
-
\big[
\hat{\mathcal{T}}_{t,N}- \mathcal{T}_t
\big]
(V_{t+1})
\big \|_\infty
= o_p(N^{-1/2}).
\end{align}
Then $N^{1/2}(\hat{V}_{t,N} - V_{t})$ converges in distribution to mean-zero Gaussian process $\mathfrak{G}_{t}$ in  $C(\mathcal{X}_{t})$ with
\begin{align}
\label{eq:clt}
\mathfrak{G}_{t}(\cdot)=
\mathbb{E}_{\xi_t \sim P_t}
[
\mathfrak{G}_{t+1}(F_t(\cdot, \pi^*_t(\cdot), \xi_t))
]
+
\mathfrak{H}_{t}(\cdot),
\end{align}
and covariance function
\begin{align}
\label{eq:covariance-function}
\begin{aligned}
\Gamma_t(x_t, x_t^\prime)
 & =
\cov(
\Psi_{t+1}(x_{t}, \pi^*_t(x_t), \mathfrak{G}_{t+1}),
\Psi_{t+1}(x_{t}^\prime, \pi^*_t(x_t^\prime), \mathfrak{G}_{t+1})
)\\
& \quad +
\cov(
\Phi_t(x_t, \pi^*_t(x_t), \xi_t),
\Phi_t(x_t^\prime,  \pi^*_t(x_t^\prime), \xi_t)
).
\end{aligned}
\end{align}
\end{theorem}

\begin{remark}
\normalfont
Let us define the operator
$\mathcal{K}_t \colon C(\mathcal{X}_{t+1}) \to C(\mathcal{X}_t)$
by
\begin{equation}
\label{eq:def-K_t}
[\mathcal{K}_t W](x_t)
\coloneqq
\mathbb{E}_{\xi_t \sim P_t}
\big[
W(F_t(x_t,\pi^*_t(x_t),\xi_t))
\big].
\end{equation}
Under the hypotheses of \Cref{thm:clt},
$\mathcal{K}_t$ is linear and bounded. Hence, the recursion
in \eqref{eq:clt} may be written equivalently as
\begin{equation}\label{recurk}
\mathfrak{G}_t
=
\mathcal{K}_t \mathfrak{G}_{t+1}
+
\mathfrak{H}_t.
\end{equation}
This formulation helps interpret \eqref{eq:clt}: the operator
$\mathcal{K}_t$ propagates the next-stage fluctuation
$\mathfrak{G}_{t+1}$ backward to stage $t$ by averaging it over the
current-stage noise $\xi_t$ after composition with the state dynamics
under the optimal policy.
Moreover, the above recursion may be interpreted as a dynamic
programming principle for the limit laws.
\hfill $\square$
\end{remark}

We prepare our  proof of \Cref{thm:clt}.
 Our next result verifies that the Bellman operator $\mathcal{T}_t$
is Hadamard directionally differentiable and hence
verifies \Cref{assume:general}(i).
We define the action-value operator $\mathcal{Q}_t \colon C(\mathcal{X}_{t+1}) \to C(\mathcal{X}_t  \times \mathcal{U}_t)$ by
\begin{align}
\label{eq:q-function}
\begin{aligned}
[\mathcal Q_t(W)](x_t,u_t)
& \coloneqq
\mathbb E_{\xi_t \sim P_t}\big[
f_t(x_t,u_t,\xi_t)+W(F_t(x_t,u_t,\xi_t))
\big].
\end{aligned}
\end{align}

\begin{lemma}
\label{lem:composition-is-hadamard}
Let $t \in \{1, \ldots, T\}$.
If \Cref{ass-1,ass-2} hold, then
$\mathcal{Q}_t$ and  $\mathcal{T}_t$ are Hadamard 
differentiable at $V_{t+1}$, 
$[\mathcal{Q}_t'(V_{t+1}; W)](x_t, u_t) = \mathbb{E}_{\xi_t \sim P_t}
[W(F_t(x_t, u_t, \xi_t))]$, $W\in C(\X_{t+1})$, and
\begin{equation}\label{derivbell-1}
[\mathcal{T}_t^\prime(V_{t+1}; W)](x_t)
=
\mathbb{E}_{\xi_t \sim P_t}[W(F_t(x_t, \pi^*_t(x_t), \xi_t))],
\quad W\in C(\X_{t+1}).
\end{equation}
\end{lemma}

\begin{proof}
The mapping $\mathcal{T}_t$ and $\mathcal{Q}_t$ are Lipschitz continuous.
Since $\mathcal{Q}_t$ is affine and the linear term is bounded, 
it is Hadamard directionally differentiable at $V_{t+1}$
with $\mathcal{Q}_t'(V_{t+1}; W) = \mathbb{E}_{\xi_t \sim P_t}
[W(F_t(\cdot, \cdot, \xi_t))]$.
Moreover, 
$G_t \colon C(\mathcal{X}_t \times \mathcal{U}_t) \to
C(\mathcal{X}_t)$
defined by
$[G_t(\psi)](x_t) \coloneqq \inf_{u_t \in \mathcal{U}_t}\,
\psi(x_t, u_t)
$
from \Cref{lem:Gt-Hadamard}
is Lipschitz continuous and directionally differentiable
with derivative given in \eqref{dirderiv}.
Since $\mathcal{T}_t(V)  = G_t(\mathcal{Q}_t(V))$,
the chain rule  implies
that $\mathcal{T}_t$ is directionally differentiable at $V_{t+1}$
and provides the derivative formula.
\end{proof}

\begin{proof}[{Proof of \Cref{thm:clt}}]
We verify the hypotheses of \Cref{thm:general-clt}.
\Cref{lem:composition-is-hadamard} ensures that $\mathcal{T}_t$
is Hadamard  differentiable at $V_{t+1}$, 
and $\mathcal{T}_t$ and $\hat{\mathcal{T}}_{t,N}$ are Lipschitz continuous.
Since 
the value function $V_{t+1}$ is Lipschitz continuous
(cf., \cite[Proposition 1]{Bertsekas1975}),
the convergence in distribution in \eqref{eq:clt-mathcalV}
which corresponds to 
\eqref{eq:general-intermediate-limit} in  \Cref{thm:general-clt}, 
and the identity for the covariance function in \eqref{covHt}
 can be established similarly to \Cref{pr-1}.
Moreover, $\hat{\mathcal{T}}_{t,N}$ is independent of
$\hat{V}_{t+1,N}$.
Hence, \Cref{assume:general} holds true.
Now, an application of \Cref{thm:general-clt} yields \eqref{eq:clt}. 
The covariance identity \eqref{eq:covariance-function} follows
from \eqref{covHt} and \eqref{eq:clt}.
\end{proof}

We comment on the recursive structure of the asymptotic
variance of $N^{1/2}\big(\hat{V}_{t,N}(x_t) - V_t(x_t) \big)$.

\begin{remark}
\label{rem:inductive-clt}
{\rm
Consider the  covariance function $\Gamma_{t+1}$ of the process
$\mathfrak{G}_{t+1}$ from \eqref{eq:covariance-function}.
Since the covariance operator is bilinear, and the expectation
operator
is linear,
 we can write for all $x_t$, $x_t'\in \mathcal{X}_t$,
\begin{align*}
&\cov(
\Psi_{t+1}(x_{t}, \pi^*_t(x_t), \mathfrak{G}_{t+1}),
\Psi_{t+1}(x_{t}^\prime, \pi^*_t(x_t^\prime), \mathfrak{G}_{t+1})
)
\\
& =
\mathbb{E}_{(\xi_t, \xi_t^\prime) \sim P_t \times P_t}
[\Gamma_{t+1}(F_t(x_t, \pi^*_t(x_t), \xi_t), F_t(x_t^\prime,
\pi^*_t(x_t^\prime), \xi_t^\prime))
],
\end{align*}
where
$\Psi_{t+1}$ is defined in \eqref{Psi-t},
and $\xi_t'$ is independent of $\xi_t$.
In particular, this identity shows that
only the covariance operator of the limit process $\mathfrak{G}_{t+1}$
is needed for variance computations.
Therefore, the variance function of
the Gaussian process
$\mathfrak{G}_{t}$,
which is the asymptotic variance of
$N^{1/2}\big(\hat{V}_{t,N}(x_t) - V_t(x_t) \big)$,
 can be written as
\begin{align}
\label{asymvarin}
\begin{aligned}
\var[\mathfrak{G}_{t}(x_t)]
& =
\var\big[f_t(x_t,\pi^*_t(x_t),\xi_{t})+
V_{t+1} (F_t(x_t,\pi^*_t(x_t),\xi_{t}))\big]
\\
& \quad +
 \var\left(\mathbb{E}_{\xi_t\sim P_t}[\mathfrak{G}_{t+1}(F_t(x_t, \pi^*_t(x_t), \xi_t))]\right ),
\end{aligned}
\end{align}
where
\begin{align}
\label{varformula}
 \var\left(\mathbb{E}_{\xi_t\sim P_t}[\mathfrak{G}_{t+1}(F_t(x_t, \pi^*_t(x_t), \xi_t))]\right )
& = \mathbb{E}_{(\xi_t, \xi_t^\prime) \sim P_t \times P_t}
[\Gamma_{t+1}(F_t(x_t, \pi^*_t(x_t), \xi_t), F_t(x_t,
\pi^*_t(x_t), \xi_t^\prime))].
\end{align}
  The limit variance  in \eqref{asymvarin} consists of two components.
The first is the variance of the sum of the current stage cost and the
future cost-to-go function, where the variance is taken over the current noise
 $\xi_t \sim P_t$.
We refer to this term as \emph{current stage variance.}
The second component, given in \eqref{varformula},  propagates uncertainty
from time $t+1$ backward to time $t$.
We refer to the   term  \eqref{varformula} as the {\em propagated variance}.
 It is the variance, taken over
future randomness ($\xi_{t+1}, \xi_{t+2}, \dots, \xi_T$), of the conditional
expectation over the current noise $\xi_t$ of the future limit
process $\mathfrak{G}_{t+1}$. This limit process is
evaluated at $F_t(x_t, \pi^*_t(x_t), \xi_t)$.
This second variance term propagates  the limit
distribution of the error $N^{1/2} (\hat{V}_{t+1,N} -V_{t+1})$ from time
$t+1$ to $t$.
Crucially, this backward induction does
not require knowing the mean zero Gaussian limit distribution $\mathfrak{G}_{t+1}$, but only
its covariance function $\Gamma_{t+1}(x_{t+1}, x_{t+1}^\prime)$.
}$\hfill \square$
\end{remark}

\begin{remark}
\label{rem:G1}
\normalfont
Let  \Cref{ass-1,ass-1',ass-2} as well as
\eqref{eq:continuity} hold for all
$t \in \{1, \ldots, T-1\}$.
For $x_1 \in \mathcal{X}_1$,
we define $\boldsymbol{x}_1 \coloneqq x_1$.
For $t \in \{1, \ldots, T\}$, $\boldsymbol{x}_{t+1} = \boldsymbol{x}_{t+1}(\xi_{[t]})$ denotes the random state at time $t+1$ generated by the optimal policy, where
$\xi_{[t]}$ is the history of the random process. Recall that under \Cref{ass-1},   the state process corresponding to the optimal policy  can be considered as a function of the history of the random data process.
Iterating the recursion \eqref{eq:clt} yields for all $x_1 \in \mathcal{X}_1$, 
\begin{equation}\label{valuesum}
\mathfrak{G}_1(x_1)
=
\tsum_{t=1}^T
\mathbb{E}_{\xi_{[t-1]}\sim P_{[t-1]}}
\big[
\mathfrak{H}_t\big(\boldsymbol{x}_t(\xi_{[t-1]})\big) \mid \mathfrak{H}_t
\big].
\end{equation}
Recall that $\mathfrak{H}_t$ is the mean-zero Gaussian process in $C(\X_t)$  with covariance function \eqref{covHt}. 
$\hfill \square$
\end{remark}

\subsection{Central limit theorems for SAA optimal values}
\label{sec-saaoptval}

In this section, we prove that
$N^{1/2}(\hat{V}_{1,N}(x_1) - V_1(x_1))$ converges in distribution to a 
normal random variable,
and we derive an explicit formula for the asymptotic variance.
From \Cref{rem:G1}, we recall that $\boldsymbol{x}_{t+1} = \boldsymbol{x}_{t+1}(\xi_{[t]})$ denotes the random state at time $t+1$ generated by the optimal policy.
For $t \in \{1, \ldots, T\}$, we define 
the product probability distribution $P_{[t]} \coloneqq P_1 \times \cdots \times P_t$.

\begin{proposition}
\label{prop:CLT-SAA-optimal-value}
Let \Cref{ass-1,ass-1',ass-2} hold, and suppose that
\eqref{eq:continuity} is satisfied for all
$t \in \{1,\ldots,T-1\}$. 
If $x_1 \in \mathcal{X}_1$,
then
$
N^{1/2}\bigl(\hat V_{1,N}(x_1)-V_1(x_1)\bigr)
$
converges in distribution 
to a mean-zero Gaussian random variable 
$\mathfrak G_1(x_1)$ with variance
\begin{align}
\label{saaov-5'}
\var(\mathfrak G_1(x_1))
&=
\tsum_{t=1}^{T}
\var_{\xi_t \sim P_t}\Big(
\mathbb E_{\xi_{[t-1]} \sim P_{[t-1]}}
\Big[
f_t(\boldsymbol x_t,\pi^*_t(\boldsymbol x_t),\xi_t)
+
V_{t+1}(F_t(\boldsymbol x_t,\pi^*_t(\boldsymbol x_t),\xi_t))
\mid \xi_t \Big]
\Big).
\end{align}
\end{proposition}

Before we establish \Cref{prop:CLT-SAA-optimal-value}, let us comment
on the variance identity in \eqref{saaov-5'}. We define
\begin{align}\label{totalsum}
Z_t(\xi_1, \ldots, \xi_t) \coloneqq 
f_t(\boldsymbol x_t,\pi^*_t(\boldsymbol x_t),\xi_t)
+
V_{t+1}(F_t(\boldsymbol x_t,\pi^*_t(\boldsymbol x_t),\xi_t)).
\end{align}
In  \eqref{saaov-5'}, we first compute for  $\xi_t \in \Xi_t$,
the expectation of $Z_t(\xi_1, \ldots, \xi_t)$
with respect to $P_{[t-1]}$,
yielding a random variable which only depends on $\xi_t \in \Xi_t$.
In the second step, \eqref{saaov-5'} computes the variance of this
random variable. Finally, the variance terms are summed up.

\begin{proof}[Proof of \Cref{prop:CLT-SAA-optimal-value}]
 \Cref{thm:clt,pr-1} ensure
$
N^{1/2}\bigl(\hat V_{1,N}(x_1)-V_1(x_1)\bigr)
\dst \mathfrak G_1(x_1)
$,
and $\mathfrak G_1(x_1)$ is a mean-zero Gaussian random variable.
Using \Cref{rem:G1}, 
and the independence of $\mathfrak{H}_t$,
for $t=1, \ldots, T$, owing to \Cref{ass-1,ass-1'}, we have
\[
\var(\mathfrak G_1(x_1))
=
\tsum_{t=1}^T \var(Y_t),
\quad \text{where} \quad 
Y_t \coloneqq \mathbb{E}_{\xi_{[t-1]}\sim P_{[t-1]}}
\big[
\mathfrak{H}_t(\boldsymbol{x}_t)
\mid \mathfrak{H}_t
\big].
\]

Fix $t \in \{1,\ldots,T\}$. Let
$\xi_{[t-1]}' \sim P_{[t-1]}$ be an independent copy of
$\xi_{[t-1]}$, and let $\boldsymbol{x}'_t$ be the corresponding
state at stage $t$. Since $\mathfrak H_t$ is centered,
\begin{align*}
\var(Y_t)
&=
\mathbb E_{\xi_{[t-1]}\sim P_{[t-1]},\xi_{[t-1]}'\sim P_{[t-1]}}
\Big[
\cov_{\mathfrak H_t}\big(
\mathfrak H_t(\boldsymbol x_t),
\mathfrak H_t(\boldsymbol{x}'_t)
\big)
\Big].
\end{align*}
By \Cref{thm:clt,pr-1}, the covariance kernel of $\mathfrak H_t$ is
given by
\[
\cov_{\mathfrak H_t}\big(\mathfrak H_t(x),\mathfrak H_t(x')\big)
=
\cov_{\xi_t\sim P_t}\big(
\Phi_t(x,\pi^*_t(x),\xi_t),
\Phi_t(x',\pi^*_t(x'),\xi_t)
\big).
\]
Hence,
\begin{align*}
\var(Y_t)
&=
\mathbb E_{\xi_{[t-1]}\sim P_{[t-1]},\xi_{[t-1]}'\sim P_{[t-1]}}
\Big[
\cov_{\xi_t\sim P_t}\big(
\Phi_t(\boldsymbol x_t,\pi^*_t(\boldsymbol x_t),\xi_t),
\Phi_t(\boldsymbol{x}'_t,\pi^*_t(\boldsymbol{x}'_t),\xi_t)
\big)
\Big].
\end{align*}
Expanding this covariance, 
we obtain
\begin{align*}
& \cov_{\xi_t\sim P_t}\big(
\Phi_t(\boldsymbol x_t,\pi^*_t(\boldsymbol x_t),\xi_t),
\Phi_t(\boldsymbol{x}'_t,\pi^*_t(\boldsymbol{x}'_t),\xi_t)
\big)
\\
&\quad = 
\mathbb{E}_{\xi_t \sim P_t}[
\Phi_t(\boldsymbol x_t,\pi^*_t(\boldsymbol x_t),\xi_t)
\Phi_t(\boldsymbol{x}'_t,\pi^*_t(\boldsymbol{x}'_t),\xi_t)
\mid \xi_{[t-1]}, \xi_{[t-1]}'
]
\\
& \quad  \quad - 
\mathbb{E}_{\xi_t \sim P_t}[
\Phi_t(\boldsymbol x_t,\pi^*_t(\boldsymbol x_t),\xi_t)
\mid \xi_{[t-1]}, \xi_{[t-1]}'
]
\mathbb{E}_{\xi_t \sim P_t}[
\Phi_t(\boldsymbol{x}'_t,\pi^*_t(\boldsymbol{x}'_t),\xi_t)
\mid \xi_{[t-1]}, \xi_{[t-1]}'
].
\end{align*}
Using the fact that
$\xi_t$ is independent of $\xi_{[t-1]}$ and $\xi_{[t-1]}'$, we obtain
\begin{align*}
\var(Y_t)
&=
\var_{\xi_t\sim P_t}\big(
\mathbb E_{\xi_{[t-1]}\sim P_{[t-1]}}
\big[
\Phi_t(\boldsymbol x_t,\pi^*_t(\boldsymbol x_t),\xi_t) \mid \xi_t
\big]
\big).
\end{align*}
Summing over $t=1,\ldots,T$ proves \eqref{saaov-5'}.
\end{proof}

The variance representation in \eqref{saaov-5'} also admits a useful
comparison with the variance of the realized cumulative cost under the
optimal policy.

\begin{proposition}
\label{prop:sqrtT-scaling}
Under the hypotheses of \Cref{prop:CLT-SAA-optimal-value},  we have
\begin{align}
\label{eq:variance-sum}
\begin{aligned}
\var(\mathfrak{G}_1(x_1))
\leq 
\var\big(\tsum_{t=1}^{T} f_t(\bx_t,\pi^*_t(\bx_t),\xi_t)+f_{T+1}(\bx_{T+1})\big).
\end{aligned}
\end{align}
\end{proposition}

\begin{proof}
Fix $t\in\{1,\ldots,T\}$.
We define
$
\Delta_t \coloneqq \Delta_t(\xi_1, \ldots, \xi_{t}) \coloneqq f_t(\bx_t,\pi^*_t(\bx_t),\xi_t)+V_{t+1}(\bx_{t+1})-V_t(\bx_t)
$,
where we recall that  $\bx_t$ is a function of
$\xi_{[t-1]}$, but does not depend on $\xi_t$.
We have
\begin{align*}
\bbe_{\xi_{[t]}\sim P_{[t]}}\!\left[f_t(\bx_t,\pi^*_t(\bx_t),\xi_t)+V_{t+1}(\bx_{t+1})\mid \xi_{[t-1]}\right]
& = \bbe_{\xi_{[t]}\sim P_{[t]}}[V_t(\bx_t)\mid \xi_{[t-1]}] = V_t(\bx_t).
\end{align*}
Hence $\bbe_{\xi_{[t]}\sim P_{[t]}}[\Delta_t\mid \xi_{[t-1]}]=0$. For $s<t$,
$\bbe_{\xi_{[t]}\sim P_{[t]}}[\Delta_t\Delta_s]=\bbe_{\xi_{[t-1]}\sim P_{[t-1]}}\!\big[\Delta_s\,\bbe_{\xi_{t}\sim P_{t}}(\Delta_t\mid \xi_{[t-1]})\big]=0$.
Therefore, $\Delta_t$, $t=1, \ldots, T$, are uncorrelated.
Applying variances to the identity
\[
\tsum_{t=1}^{T} \Delta_t
=\sum_{t=1}^{T} f_t(\bx_t,\pi^*_t(\bx_t),\xi_t)+f_{T+1}(\bx_{T+1})-V_1(x_1),
\]
we obtain
\begin{align}
\label{eq:varianceofZt}
\tsum_{t=1}^{T} \var(\Delta_t)
=
\var\big(\tsum_{t=1}^{T} \Delta_t\big)
= 
\var\big(\tsum_{t=1}^{T} f_t(\bx_t,\pi^*_t(\bx_t),\xi_t)+f_{T+1}(\bx_{T+1})\big).
\end{align}

 By stagewise independence, $\xi_t$ is
independent of $\xi_{[t-1]}$. Therefore, the law of total variance
gives
\begin{align}
\label{eq:var-Zt-ltv}
\var(\Delta_t)
&=
\var_{\xi_t\sim P_t}\!\big(
\mathbb E_{\xi_{[t-1]}\sim P_{[t-1]}}[\Delta_t \mid \xi_t]
\big)
+
\mathbb E_{\xi_t\sim P_t}\!\big[
\var_{\xi_{[t-1]}\sim P_{[t-1]}}
\big(\Delta_t \mid \xi_t\big)
\big].
\end{align}
Next, note that
\[
\mathbb E_{\xi_{[t-1]}\sim P_{[t-1]}}[\Delta_t \mid \xi_t]
=
\mathbb E_{\xi_{[t-1]}\sim P_{[t-1]}}
\big[
f_t(\bx_t,\pi^*_t(\bx_t),\xi_t)+V_{t+1}(\bx_{t+1})
\mid \xi_t
\big]
-
\mathbb E_{\xi_{[t-1]}\sim P_{[t-1]}}[V_t(\bx_t)].
\]
The second term does not depend on $\xi_t$, and hence
\begin{align*}
\var_{\xi_t\sim P_t}\!\big(
\mathbb E_{\xi_{[t-1]}\sim P_{[t-1]}}[\Delta_t \mid \xi_t]
\big)
&=
\var_{\xi_t\sim P_t}\!\big(
\mathbb E_{\xi_{[t-1]}\sim P_{[t-1]}}
\big[
f_t(\bx_t,\pi^*_t(\bx_t),\xi_t)
+V_{t+1}(\bx_{t+1}) \mid \xi_t
\big]
\big).
\end{align*}
Combined with \eqref{saaov-5'}, \eqref{eq:varianceofZt}, and
\eqref{eq:var-Zt-ltv}, we obtain the
inequality in \eqref{eq:variance-sum}.
\end{proof}

The inequality in \eqref{eq:variance-sum} can be strict, as the
following example shows.

\begin{example}
\label{exam-contr}
\normalfont
Suppose that \Cref{ass-1,ass-1'} hold.
We provide an example where the inequality in 
\eqref{eq:variance-sum} is strict.
We consider  \(T=2\), 
$
\mathcal{U}_1=\mathcal{U}_2=\{0\}
$,
$
\mathcal{X}_1=\{0\}
$,
$
\mathcal{X}_2=\mathcal{X}_3=\{-1,1\}
$,
$
\Xi_1=\Xi_2=\{-1,1\}
$,
$x_1 = 0$, $F_1(x_1, u_1, \xi_1) = \xi_1$,
and $F_2(x_2, u_2, \xi_2) = \xi_2$,
where \(\xi_1\) and \(\xi_2\) are independent Rademacher random
variables. Moreover, let $a,b \in \mathbb{R}$. We define  the stage costs by
$f_1(x_1,u_1,\xi_1)\coloneqq 0$,
$f_2(x_2,u_2,\xi_2)\coloneqq a \xi_2 +  bx_2\xi_2$,
$f_3(x_3)\coloneqq 0$.
Then $x_2=F_1(x_1,u_1,\xi_1)=\xi_1$, and the realized total cost is
\[
f_1(x_1,\pi^*_1(x_1),\xi_1)
+
f_2(\boldsymbol{x}_2,\pi^*_2(\boldsymbol{x}_2),\xi_2)
+
f_3(\boldsymbol{x}_3)
=
a\xi_2 + b\xi_1\xi_2,
\]
and hence
\[
\var\big(
f_1(x_1,\pi^*_1(x_1),\xi_1)
+
f_2(\boldsymbol{x}_2,\pi^*_2(\boldsymbol{x}_2),\xi_2)
+
f_3(\boldsymbol{x}_3)
\big)
=
a^2 + b^2.
\]

The true value functions are given by
\[
V_3\equiv 0,
\quad
V_2(x_2)=\mathbb{E}[f_2(x_2,0,\xi_2)]
=
\mathbb{E}[a\xi_2 + b x_2\xi_2]
=
0,
\quad
\text{and} 
\quad 
V_1(x_1)=0.
\]
The stage-\(2\) SAA value function is
$
\hat V_{2,N}(x_2)
=
\tfrac{1}{N}\tsum_{i=1}^N f_2(x_2,0,\xi_{2i})
=
(a+bx_2)\bar{\xi}_{2,N}
$,
where $ \bar{\xi}_{2,N}\coloneqq \frac{1}{N}\sum_{i=1}^N \xi_{2i}$. At stage \(1\),
\[
\hat V_{1,N}(x_1)
=
\tfrac{1}{N}\tsum_{i=1}^N
\hat V_{2,N}\bigl(F_1(x_1,0,\xi_{1i})\bigr)
=
(a+b\bar{\xi}_{1,N})\bar{\xi}_{2,N},
\quad \text{where} \quad 
\bar{\xi}_{1,N}\coloneqq \tfrac{1}{N}\tsum_{i=1}^N \xi_{1i}.
\]
Hence
$
N^{1/2}\bigl(\hat V_{1,N}(x_1)-V_1(x_1)\bigr)
= aN^{1/2}\bar{\xi}_{2,N}
+ bN^{1/2}\bar{\xi}_{1,N} \bar{\xi}_{2,N}
\dst aZ
$,
where $Z \sim \mathcal{N}(0,1)$. 
Hence, the left-hand side in 
\eqref{eq:variance-sum} is $a^2$, 
while the right-hand side in \eqref{eq:variance-sum} 
is $a^2+b^2$. In particular, for $b \neq 0$, 
the inequality in \eqref{eq:variance-sum} is strict.
\hfill $\square$
\end{example}

\subsection{Sufficient conditions for the stochastic equicontinuity-type condition}
\label{subsect:equicontinuity}

We provide conditions that are sufficient to ensure the stochastic equicontinuity-type condition \eqref{eq:continuity}.

\subsubsection{Lipschitz continuity-type conditions}

Motivated by \Cref{rem:Bellman-Lipschitz-1}, 
this section  provides a Lipschitz continuity-type condition that is sufficient to ensure the stochastic equicontinuity-type condition \eqref{eq:continuity} holds.
Relatedly, the proof of
\cite[Corollary 2.2]{Newey1991} demonstrates that  Lipschitz continuity-type
conditions can  imply stochastic equicontinuity.

\begin{remark}
\label{lem:Bellman-Lipschitz-1}
\normalfont
Since $N^{1/2}(\hat{V}_{t+1,N} - V_{t+1}) \dst \mathfrak{G}_{t+1}$
implies that
$N^{1/2}\|\hat{V}_{t+1,N}-V_{t+1}\|_\infty = O_p(1)$,
a sufficient condition for \eqref{eq:continuity} to hold is given by
\begin{align*}
\big\|\big[\hat{\mathcal{T}}_{t,N}
-
\mathcal{T}_{t}\big](\hat{V}_{t+1,N})
-
\big[
\hat{\mathcal{T}}_{t,N}- \mathcal{T}_t
\big]
(V_{t+1})
\big\|_\infty
=
o_p(1) \|\hat{V}_{t+1,N}-V_{t+1}\|_\infty.
\end{align*}
See \Cref{rem:Bellman-Lipschitz-1} for further details.
$\hfill \square$
\end{remark}

\subsubsection{Empirical-process conditions for stochastic equicontinuity}
\label{subsubsect:Random-empirical-processes}

To verify the stochastic equicontinuity condition, it is convenient to
work one level before the minimization performed by the dynamic
programming operators, and study the associated
action-value operators directly. At that level, the difference between
the sample-based and population operators can be written as a random
empirical fluctuation indexed by states, controls, and value
functions. This makes it possible to apply standard empirical-process
tools, as we highlight below.

Recalling the definition of the action-value operator
$\mathcal{Q}_t$ in \eqref{eq:q-function}, we define its empirical
counterpart
$\hat{\mathcal Q}_{t,N} \colon C(\mathcal{X}_{t+1}) \to
C(\mathcal{X}_t \times \mathcal{U}_t)$ by
\begin{align}
\label{eq:hatq-function}
\begin{aligned}
[\hat{\mathcal Q}_{t,N}(W)](x_t,u_t)
& \coloneqq
\frac1N\sum_{i=1}^N
\big[
f_t(x_t,u_t,\xi_{ti})+W(F_t(x_t,u_t,\xi_{ti}))
\big].
\end{aligned}
\end{align}
For any measurable function $g \colon \Xi_t \to \mathbb{R}$, 
we define
\begin{align}
\label{eq:empirical-process-q}
P_t g \coloneqq \mathbb{E}_{\xi_t \sim P_t}[g(\xi_t)]
\quad \text{and}
\quad 
\hat{P}_{t, N} g
\coloneqq \mathbb{E}_{\xi_t \sim \hat{P}_{t,N}}[g(\xi_t)],
\end{align}
where $\hat{P}_{t,N}$ is given in \eqref{empirmes}.
For $(x_t, u_t, W) \in \mathcal{X}_t \times \mathcal{U}_t \times C(\mathcal{X}_{t+1})$, we define
\begin{align*}
g_{(x_t, u_t), W}(\xi_t) \coloneqq f_t(x_t, u_t, \xi_t) + W(F_t(x_t, u_t, \xi_t)).
\end{align*}
We have
$
[\mathcal Q_t(W)](x_t,u_t)=P_t g_{(x_t,u_t),W}
$
and
$
[\hat{\mathcal Q}_{t,N}(W)](x_t,u_t)=\hat P_{t,N} g_{(x_t,u_t),W}
$.
Now, we
consider the empirical process
$
\mathbb G_{t,N} \coloneqq  N^{1/2}(\hat{P}_{t,N} - P_t)
$
indexed by the function class 
\[
\{g_{(x_t, u_t), W} \colon (x_t, u_t) \in \mathcal{X}_t \times \mathcal{U}_t, 
W \in C(\mathcal{X}_{t+1})\}.
\]
Hence for all $(x_t, u_t) \in \mathcal{X}_t \times \mathcal{U}_t$, 
	\begin{align*} 
N^{1/2}\big( [\hat{\mathcal Q}_{t,N}-\mathcal Q_t](\hat V_{t+1,N}) - [\hat{\mathcal Q}_{t,N}-\mathcal Q_t](V_{t+1}) \big)(x_t,u_t) 
= \mathbb{G}_{t,N}\big(g_{(x_t, u_t), \hat{V}_{t+1,N}}-g_{(x_t, u_t), V_{t+1}}\big).
 \end{align*}
In particular,
\begin{equation}
\label{eq:q-equicontinuity}
\big\| \big[\hat{\mathcal{Q}}_{t,N} - \mathcal{Q}_{t}\big](\hat{V}_{t+1,N}) - \big[\hat{\mathcal{Q}}_{t,N}- \mathcal{Q}_t\big](V_{t+1}) \big\|_\infty = o_p(N^{-1/2}).
\end{equation}
is equivalent to
\begin{equation}
\label{eq:stochastic-equicontinuity-target}
\sup_{(x_t,u_t) \in \mathcal{X}_t \times \mathcal{U}_t}\big| \mathbb{G}_{t,N}\big(g_{(x_t,u_t), \hat{V}_{t+1,N}} - g_{(x_t,u_t), V_{t+1}}\big) \big|  = o_p(1).
\end{equation}

Section~3.13 in \cite{Vaart2023} derives sufficient conditions for
\eqref{eq:stochastic-equicontinuity-target} to hold,
and \eqref{eq:stochastic-equicontinuity-target} is an instance of \cite[eq.\ (3.13.2)]{Vaart2023}.

Let us first relate the 
stochastic equicontinuity-type conditions
\eqref{eq:continuity} and \eqref{eq:q-equicontinuity}.

\begin{proposition}
\label{lem:equicontinuity-equivalence}
Let $t \in \{1, \ldots, T\}$.
If  \Cref{ass-1,ass-1',ass-2} as well as $N^{1/2}(\hat{V}_{t+1,N} - V_{t+1}) \dst \mathfrak{G}_{t+1}$
hold, then the stochastic equicontinuity-type condition 
\eqref{eq:q-equicontinuity} implies  \eqref{eq:continuity}.
\end{proposition}

\begin{proof}
The proof of \Cref{lem:equicontinuity-equivalence} is provided in \Cref{sect:appendix-Random-empirical-processes}.
\end{proof}

Under \Cref{ass-2}, $f_{T+1}$ is bounded and Lipschitz continuous.
Let $\ell_{T+1}$ be its Lipschitz constant, and 
let $r_{T+1}$ be its supremum norm.
 \Cref{ass-2} also ensures 
\begin{align}
\label{eq:Bt}
e_t\coloneqq
\sup_{(x_t,u_t,\xi_t)\in
\X_t\times\U_t\times\Xi_t}
|f_t(x_t,u_t,\xi_t)| < \infty.
\end{align}

\begin{lemma}\label{LipschitzII}
Let \Cref{ass-1,ass-1',ass-2} hold.
For \(N \in \mathbb{N}\), let
$
\hat \ell_{T+1,N}\coloneqq \ell_{T+1},
$
and, for \(t=T,\ldots,1\),
\[
 \ell_t\coloneqq \mathbb E_{\xi_t\sim P_t}[K_t(\xi_t)](1+\ell_{t+1}),
\quad
\hat \ell_{t,N}\coloneqq \mathbb E_{\xi_t\sim \hat{P}_{t,N}}[K_t(\xi_t)](1+\hat \ell_{t+1,N}),
\quad \text{and} \quad r_t\coloneqq e_t+r_{t+1}.
\]
Then for \(t=1,\ldots,T\):
{\rm (i)} The value function \(V_t\) is Lipschitz continuous 
with Lipschitz constant $\ell_t$, 
and its supremum norm is bounded by $r_t$.
{\rm (ii)} The SAA value function \(\hat V_{t,N}\) is Lipschitz continuous
with Lipschitz constant $\hat{\ell}_{t,N}$, and
 its supremum norm is bounded by $r_t$.
{\rm (iii)} As $N \to \infty$,
with probability one,
$
\hat \ell_{t,N}\to \ell_t
$.
In particular, with probability one, for all sufficiently large \(N \in \mathbb{N}\),
\[
\hat \ell_{t,N}\le \ell_t+1.
\]
\end{lemma}
\begin{proof}
This follows from standard computations, and the strong Law of Large Numbers. 
\end{proof}

Let $\cV_{t+1} \subset C(\mathcal{X}_{t+1})$
be  the  set of functions that are 
 Lipschitz continuous with  Lipschitz constant $\ell_{t+1}+1$, and 
uniformly bounded by
$r_{t+1}$.

\begin{lemma}
\label{lem:SAA-value-function-inclusion}
If $t \in \{1, \ldots, T\}$, and \Cref{ass-1,ass-1',ass-2} hold,
then, with probability one, for all sufficiently large $N \in \mathbb{N}$, 
$\hat{V}_{t+1, N} \in \cV_{t+1}$.
\end{lemma}
\begin{proof}
An application of \Cref{LipschitzII} yields the assertion.
\end{proof}

Our next assumption ensures that the class $\cV_{t+1}$ has sufficiently
small metric entropy,
where we recall that  $\mathcal{H}(\varepsilon, \mathcal{Y}_0)$
denotes the $\varepsilon$-metric entropy number of $\mathcal{Y}_0$.

\begin{assumption}
\label{ass-4}
For $t=1, \ldots, T$,
$
\int_{0}^1 (\mathcal{H}(\varepsilon, \cV_{t}))^{1/2}
\, \mathrm{d} \varepsilon < \infty
$.
\end{assumption}

The next result shows that \Cref{ass-4} provides a convenient sufficient
condition for the stochastic equicontinuity-type condition
\eqref{eq:q-equicontinuity}.

\begin{proposition}
\label{lem:bellman-clt}
Let $t \in \{1, \ldots, T\}$.
Suppose that \Cref{ass-1,ass-1',ass-2,ass-4} hold. If $N^{1/2}(\hat{V}_{t+1,N} - V_{t+1}) \dst \mathfrak{G}_{t+1}$, then the stochastic equicontinuity condition \eqref{eq:q-equicontinuity} is satisfied.
\end{proposition}

\begin{proof}
We verify the hypotheses of \cite[Theorem 3.13.4]{Vaart2023}.
Using \eqref{eq:empirical-process-q}, 
for all $(x_t, u_t, \xi_t) \in \mathcal{X}_t \times \mathcal{U}_t \times \Xi_t$,
\begin{align*}
|
g_{(x_t, u_t), \hat{V}_{t+1,N}}(\xi_t)
- g_{(x_t, u_t), V_{t+1}}(\xi_t)
|
\leq 
\|\hat{V}_{t+1,N}-V_{t+1}\|_\infty.
\end{align*}
Combined with $N^{1/2}(\hat{V}_{t+1,N} - V_{t+1}) \dst \mathfrak{G}_{t+1}$,
\begin{align*}
\sup_{(x_t, u_t) \in \mathcal{X}_t \times \mathcal{U}_t}\, 
\mathbb{E}_{\xi_t \sim P_t}[
(g_{(x_t, u_t), \hat{V}_{t+1,N}}(\xi_t)
- g_{(x_t, u_t), V_{t+1}}(\xi_t)
)^2
]
= o_p(1),
\end{align*}
which corresponds to \cite[eq.\ (3.13.3)]{Vaart2023}.
Moreover, \Cref{lem:SAA-value-function-inclusion}
ensures with probability one, for all sufficiently large $N$, 
$\hat{V}_{t+1, N} \in \cV_{t+1}$.

Next, we show that
$
\mathcal{F}_t \coloneqq \{g_{(x_t, u_t), W} \colon (x_t, u_t) \in \mathcal{X}_t \times \mathcal{U}_t, 
W \in \cV_{t+1}\}
$
is a $P_t$-Donsker class. 
We refer the reader to \cite{Vaart2023} for the definition
of Donsker classes.
We note that 
for all $\xi_t \in \Xi_t$,
$
 \mathcal{X}_t \times \mathcal{U}_t \times \cV_{t+1} \ni (x_t, u_t, w_{t+1}) \mapsto g_{(x_t, u_t), w_{t+1}}(\xi_t)$
is Lipschitz continuous 
with Lipschitz constant
$
K_t(\xi_t)\big(2+\ell_{t+1}\big) + 1
$,
where
we equip 
$\mathcal{X}_t \times \mathcal{U}_t \times C(\mathcal{X}_{t+1})$
with the norm
$
\|(x_t, u_t, w_{t+1}) \|
\coloneqq
\|x_t\| + \|u_t\| + \|w_{t+1}\|_{\infty}
$,
and $\ell_{t+1}$ is defined in \Cref{LipschitzII}.
Under \Cref{ass-2}(iii), $K_t(\xi_t)$ is square integrable.

Using \Cref{ass-4}, together with \Cref{ass-2}, which ensures that $\mathcal{X}_t \subset \mathbb{R}^{n_t}$ and $\mathcal{U}_t \subset \mathbb{R}^{m_t}$ are compact, we obtain
\begin{align*}
\int_{0}^1 (\mathcal{H}(3\varepsilon,\mathcal{X}_{t} \times \mathcal{U}_t \times  \cV_{t+1}))^{1/2}
\, \mathrm{d} \varepsilon
& \leq
\int_{0}^1 (\mathcal{H}(\varepsilon, \cV_{t+1}))^{1/2}
\, \mathrm{d} \varepsilon
\\
& \quad +
\int_{0}^1 (\mathcal{H}(\varepsilon, \mathcal{X}_{t}))^{1/2}
\, \mathrm{d} \varepsilon
+
\int_{0}^1 (\mathcal{H}(\varepsilon, \mathcal{U}_t))^{1/2}
\, \mathrm{d} \varepsilon < \infty.
\end{align*}
Combined with \cite[Theorems 2.5.6 and 2.7.17]{Vaart2023}, 
$\mathcal{F}_t$ is $P_t$-Donsker.

Now,  \cite[Theorem~3.13.4]{Vaart2023}
ensures \eqref{eq:stochastic-equicontinuity-target},
which is equivalent  to \eqref{eq:q-equicontinuity}.
\end{proof}

\Cref{ass-4} is satisfied under certain conditions. We discuss two of
them next.

\begin{remark}
\label{rem:covering-numbers}
\normalfont
Let \Cref{ass-2} hold.
For $t = 1, \ldots, T$, let the state spaces $\mathcal{X}_t$ be convex
with nonempty interior.
{\rm (i)} If $n_t=1$,  $t = 1, \ldots, T$,
then
we have
(see, e.g.,  \cite[Theorem~2.7.1]{Vaart2023}),
\begin{align*}
\mathcal{H}(\varepsilon, \cV_{t})
\leq C(1/\varepsilon)^{n_{t}}
\quad \text{for all} \quad \varepsilon > 0,
\end{align*}
where $C>0$ is a constant independent of $\varepsilon$,
but possibly depending on $r_{t}$, $\ell_{t}$, and $n_t$.
We obtain that \Cref{ass-4} is satisfied.

{\rm (ii)}
If $n_t \in \{1,2\}$,
$V_t$ is convex and
with probability one, $\hat{V}_{t, N}$,
$N \in \mathbb{N}$,
are
convex,  $t = 1, \ldots, T$,
we can let every element in
$\cV_{t+1}$
be convex. We have
(see, e.g.,  \cite[Corollary~2.7.15]{Vaart2023}),
\begin{align*}
\mathcal{H}(\varepsilon, \cV_{t})
\leq C (1/\varepsilon)^{n_{t}/2}
\quad \text{for all} \quad \varepsilon > 0,
\end{align*}
where the constant $C>0$ is independent of $\varepsilon$,
but it may depend on $r_{t}$, $\ell_{t}$, and $n_t$.
We obtain that \Cref{ass-4} is satisfied.
$\hfill \square$
\end{remark}

We verify \Cref{ass-4}
for an inventory control problem.

\begin{example}[{Inventory control}]
\normalfont
We consider a version of the inventory control problem (cf., \cite[Section 4.2]{Bertsekas2005},\cite{zipkin}).
Specifically, we define the system dynamics and stage cost as follows:
\begin{align*}
F_t(x_t, u_t, \xi_t) \coloneqq x_t + u_t - \xi_t,
\quad \text{and} \quad
f_t(x_t, u_t, \xi_t) \coloneqq c_t u_t + \psi_t(x_t, u_t, \xi_t),
\end{align*}
where
$\psi_t(x_t, u_t, \xi_t)
\coloneqq
b_t[\xi_t - (x_t+u_t)]_+
+
h_t[x_t + u_t - \xi_t]_+$.
Moreover, $f_{T+1}(x_{T+1}) \coloneqq 0$.
Here,
$c_t$ denotes the per-unit ordering cost,
$b_t$ is the backordering cost,
$h_t \geq 0$ is the holding cost,
and we assume $b_t > c_t > 0$.

The control sets are given by
$\mathcal{U}_t \coloneqq [0, \overline{u}_t]$
with $\overline{u}_t \in (0,\infty)$
and the initial state satisfies $x_1 \in [0, \infty)$.
The distributions of the  disturbance variables $\xi_t$ are supported on bounded intervals
$\Xi_t \coloneqq  [0, \overline{\xi}_t]$, for example, they can be
  uniformly distributed  $\xi_t \sim \mathrm{Uniform}(\Xi_t)$.
This setup allows us to define the state spaces as
$
\mathcal{X}_t = \big[x_1 - \tsum_{s=1}^{t-1} \overline{\xi}_s, \,x_1 + \tsum_{s=1}^{t-1} \overline{u}_s\big].
$
These choices ensure that the compactness conditions on $\mathcal{U}_t$, $\mathcal{X}_t$, and $\Xi_t$, as required by \Cref{ass-2}(i), are satisfied.
Combined with  \Cref{rem:covering-numbers}, we find that \Cref{ass-4} holds.
\hfill $\square$
\end{example}

\section{Central limit theorems for linear quadratic regulator}
\label{sec:LQR}

In this section, we apply \Cref{thm:general-clt} to LQR,
which is an instance of the SOC problem introduced in \Cref{sec-basic}.
We consider the LQR problem
as described in \cite[Section 4.1]{Bertsekas2005}.
It is defined by
$n_t \coloneqq n$, $m_t \coloneqq m$, and $d_t \coloneqq n$,
$F_t(x_t, u_t, \xi_t) \coloneqq A_tx_t + B_tu_t +\xi_t$,
$f_t(x_t,u_t, \xi_t) \coloneqq x_t^\top Q_t x_t + u_t^\top R_t u_t$,
where $Q_t \in \mathbb{R}^{n \times n}$
and $R_t \in \mathbb{R}^{m \times m}$ are symmetric positive definite,
and $f_{T+1}(x_{T+1}) \coloneqq x_{T+1}^\top Q_{T+1} x_{T+1}$,
where $Q_{T+1} \in \mathbb{R}^{n \times n}$ is symmetric positive definite.
Moreover, $\xi_t$ has mean zero and a finite fourth moment.

According to \cite[Section 4.1]{Bertsekas2005}, we have
$V_{t}(x_t) = x_t^\top H_t x_t + q_t$, 
$t=1, \ldots, T+1$,
where
$q_t = q_{t+1} + \mathbb{E}[\xi_t^\top H_{t+1} \xi_t]$
with $q_{T+1} \coloneqq 0$, 
 $H_{T+1} \coloneqq  Q_{T+1}$,
and
$H_t$ ($t=T, \ldots, 1$) are symmetric positive definite matrices and 
solve the discrete-time Riccati equation
(see, e.g., \cite[eqns.\ (4.4)--(4.5)]{Bertsekas2005}).
Moreover, 
\begin{align*}
K_t \coloneqq R_t + B_t^\top H_{t+1} B_t, \quad 
L_t \coloneqq -K_t^{-1} B_t^\top H_{t+1} A_t,  \quad 
\text{and} \quad 
M_t \coloneqq A_t + B_t L_t.
\end{align*}

To apply \Cref{thm:general-clt} to LQR, we first define suitable function spaces. 
We define the weight function
\[
w(x) \coloneqq 1 + \|x\|_2^2,
\]
and the weighted space of continuous functions 
\[
C_w(\mathbb{R}^n)
\coloneqq
\left\{
v \in C(\mathbb{R}^n) :
\|v\|_{C_w(\mathbb{R}^n)} < \infty
\right\},
\quad \text{where} \quad 
\|v\|_{C_w(\mathbb{R}^n)}
\coloneqq
\sup_{x \in \mathbb{R}^n}
\frac{|v(x)|}{1 + \|x\|_2^2}.
\]
Now, we define the state space $\mathcal{C} \coloneqq \mathcal{C}_t \coloneqq 
\{v \in C_w(\mathbb{R}^n) \colon v \text{ is a quadratic polynomial}\}$, 
$t = 1, \ldots, T+1$.
Since $\mathcal{C}$  is a finite-dimensional subspace of $C_w(\mathbb{R}^n)$, it is a
separable Banach space. 
We define $\mathcal{D}_t \coloneqq \{v \in \mathcal{C} \colon \nabla^2 v(0) = 2 H_t\}$,
which is affine-isomorphic to $\mathbb{R}^n \times \mathbb{R}$.

Let us define the covariance matrix
$\Sigma_t\coloneqq \mathbb{E}_{\xi_t \sim P_t}[\xi_t \xi_t^\top]$.

\begin{lemma}
\label{lem:appendix-2}
If \Cref{ass-1,ass-1'} hold,
then for $t \in \{1, \ldots, T\}$,
\begin{align*}
N^{1/2}
\begin{bmatrix}
\frac{1}{N} \sum_{i=1}^N \xi_{ti} \\
\frac{1}{N} \sum_{i=1}^N \xi_{ti}^\top H_{t+1}\xi_{ti}
-
\mathbb{E}[\xi_{t}^\top H_{t+1}\xi_{t}]
\end{bmatrix}
\dst
\begin{bmatrix}
\vartheta_t \\
\zeta_t
\end{bmatrix}
\sim
\mathcal{N}\bigg(
0, \begin{bmatrix}
\Sigma_t& \mathbb{E}[(\xi_t^\top H_{t+1} \xi_t) \xi_t]  \\
(\mathbb{E}[(\xi_t^\top H_{t+1} \xi_t) \xi_t])^\top&
\var(\xi_t^\top H_{t+1} \xi_t)
\end{bmatrix}
\bigg).
\end{align*}
\end{lemma}

\begin{proof}
This follows from an application of the multivariate CLT.
\end{proof}

 \Cref{thm:general-clt} implies the following statement.
We define the total (asymptotic) variance $\sigma_{t,\mathrm{asym}}^2$ by
\begin{equation}
\label{eq:lqr:limit-variance}
\sigma_{t,\mathrm{asym}}^2(x_t)
\coloneqq
\var(\mathfrak{G}_t(x_t)),
\end{equation}
and
the propagated variance $\sigmaprop^2$
and current stage variance $\sigmacurr^2$ by
\begin{align}
\label{eq:lqr:prop-curr-variance}
\begin{aligned}
\sigmaprop^2(x_t)
&\coloneqq
\var(
\mathbb{E}_{\xi_t\sim P_t}[\mathfrak{G}_{t+1}(F_t(x_t, \pi^*_t(x_t) ,\xi_t))]
),
\\
\sigmacurr^2(x_t)
& \coloneqq \var\big[f_t(x_t,\pi^*_t(x_t),\xi_{t})+
V_{t+1} (F_t(x_t,\pi^*_t(x_t),\xi_{t}))\big].
\end{aligned}
\end{align}
We have $\sigma_{t,\mathrm{asym}}^2(x_t) = \sigmaprop^2(x_t) + \sigmacurr^2(x_t)$;
cf.\ \eqref{asymvarin}.

\begin{theorem}
\label{thm:clt-lqr}
Let \Cref{ass-1,ass-1'} hold.
For $t=T, \ldots, 1$,  
$
N^{1/2}(\hat{V}_{t, N} - V_{t}) \dst \mathfrak{G}_{t}
$
in $C_w(\mathbb{R}^n)$,
the limit law $\mathfrak{G}_t$ is mean-zero Gaussian, and
$\mathfrak{G}_t(x_t) = b_t^\top x_t + c_t$, where the random vector
$(b_t, c_t)$ satisfies
\begin{align*}
b_{t} = M_{t}^\top\big(2H_{t+1}\vartheta_t + b_{t+1}\big)
\quad 
\text{and}
\quad 
c_t = \zeta_t + c_{t+1}, 
\quad 
\text{with}
\quad 
b_{T+1} = 0, \quad c_{T+1} = 0.
\end{align*}
Defining 
$S_t \coloneqq \var(b_t)$, $v_t \coloneqq \var(c_t)$, 
and $s_t \coloneqq \cov(b_t, c_t)$
with 
$S_{T+1} \coloneqq 0$,
$v_{T+1} \coloneqq 0$, and 
$s_{T+1} \coloneqq 0$,
and $\gamma_t = 2 H_{t+1} \mathbb{E}[(\xi_t^\top H_{t+1} \xi_t) \xi_t]$,
 we have
\[
\sigma_{t, \mathrm{asym}}^2(x_t) = x_t^\top S_t x_t
+ 2 s_t^\top x_t +  v_t,
\]
and
\begin{align*}
\sigmaprop^2(x_t)
& = x_t^\top M_t^\top S_{t+1} M_t x_t 
+ 2x_t^\top M_t^\top s_{t+1} + v_{t+1},
\\
\sigmacurr^2(x_t)
& = 4x_t^\top M_t^\top H_{t+1} \Sigma_tH_{t+1} M_t x_t
+ \var(\xi_t^\top H_{t+1} \xi_t) + 2x_t^\top M_t^\top \gamma_t.
\end{align*}
Moreover,  we have the backward recursions
\begin{align*}
S_t &= M_t^\top (S_{t+1} + 4H_{t+1}\Sigma_tH_{t+1}) M_t, 
\quad 
s_t = M_t^\top(s_{t+1} + \gamma_t),
\quad \text{and} \quad 
v_t = v_{t+1} + \var(\xi_t^\top H_{t+1} \xi_t).
\end{align*}
\end{theorem}

We establish \Cref{thm:clt-lqr} using the following lemma,
which verifies the hypotheses of \Cref{assume:general}. 
For \(t=1,\ldots,T\), we define the dynamic programming operator
\(\mathcal{T}_t \colon \mathcal{D}_{t+1}\to  \mathcal{C}\)
as in \eqref{saaderiv-1}, but with domain $\mathcal{D}_{t+1}$ and
image $\mathcal{C}$.
Similarly, we define
the SAA dynamic programming operator $\hat{ \mathcal{T}}_{t,N} \colon  \mathcal{D}_{t+1}  \to \mathcal{C}$ as in \eqref{saaderiv-2}, but with domain $\mathcal{D}_{t+1}$ and
image $\mathcal{C}$.

\begin{lemma}
\label{lem:lqr-assumptions}
Let \Cref{ass-1,ass-1'} hold.
We have $f_{T+1} \in \mathcal{D}_{T+1}$. Then 
for $t =1, \ldots, T$: 
\begin{enumerate}[nosep,label=\textup{(\roman*)}]
\item The operator $\mathcal{T}_t \colon \mathcal{D}_{t+1} \to \mathcal{C}$ 
is continuous, and Hadamard directionally differentiable
at $V_{t+1}$ tangentially to $\mathcal{D}_{t+1}$
with 
$[\mathcal{T}_t'(V_{t+1}; D)](x_t) = b^\top M_t x_t + c$,
where
$D(x) \coloneqq b^\top x + c$,
and $\mathcal{T}_t \mathcal{D}_{t+1} \subset \mathcal{D}_t$.
\item For each $N \in \mathbb{N} $, the operators $\hat{\mathcal{T}}_{t,N} \colon \mathcal{D}_{t+1} \to \mathcal{C}$ are continuous, and independent of $\hat{V}_{t+1, N}$, 
and $\hat{\mathcal{T}}_{t,N} \mathcal{D}_{t+1} \subset \mathcal{D}_t$  
with probability one.
\item For  the random element $\mathfrak{H}_t$  with values in $\mathcal{C}$
given by $\mathfrak{H}_t(x_t) \coloneqq 2 \vartheta_t^\top H_{t+1} M_t x_t
+
\zeta_t$, 
\begin{align}
\label{eq:lqr-intermediate-limit}
N^{1/2}
\big[
\hat{\mathcal{T}}_{t,N}- \mathcal{T}_t
\big]
(V_{t+1})
\dst \mathfrak{H}_t,
\end{align}
where $\vartheta_t$ and $\zeta_t$ are given in \Cref{lem:appendix-2}. 
\item If 
$N^{1/2}(\hat V_{t+1,N}-V_{t+1}) \dst \mathfrak G_{t+1}$, 
then the stochastic equicontinuity-type condition \eqref{eq:general-continuity} is satisfied.
\end{enumerate}

\end{lemma}

We present the proofs of \Cref{thm:clt-lqr,lem:lqr-assumptions} in 
\Cref{sec:lqr-proofs}.

\section{Numerical illustrations}
\label{sec:lqr-numerical}

This section uses the classical LQR to provide empirical illustrations of the SAA method's asymptotic properties. 
In particular, we examine the asymptotic
distribution of $ N^{1/2}(\hat{V}_{1,N}(x_1) - V_1(x_1))$
and investigate
the variance structure of
$N^{1/2}(\hat{V}_{t,N}- V_t) $,
proceeding backward in time
from $ t = T $ to $ t = 1$. Moreover, we illustrate the dependence of
the variance of $ N^{1/2}(\hat{V}_{t,N}(x_t) - V_t(x_t))$ as
a function of the time period $t$.
All computer code and simulation results are available in the repository \cite{Milz2025}.

We consider a specific instance of the LQR problem  (see \Cref{sec:LQR}) given by
\begin{align*}
f_t(x_t,u_t, \xi_t) \coloneqq x_t^2 + u_t^2,
\quad
 F_t(x_t, u_t, \xi_t) \coloneqq x_t + u_t +\xi_t,
\quad
\xi_t \sim  \mathrm{Uniform}(-\sqrt{3},\sqrt{3}),
\quad \text{and} \quad
T = 20,
\end{align*}
as well as $f_{T+1}(x_{T+1}) \coloneqq x^2_{T+1}$.

From \Cref{sec:LQR}, specifically
\eqref{eq:lqr:limit-variance},  we recall the asymptotic variance
$
\sigma_{t, \mathrm{asym}}^2(x_t) =
\var(\mathfrak{G}_t(x_t))
$,
where $\mathfrak{G}_t$ is given in \Cref{thm:clt-lqr}. 
For our LQR instance,
\Cref{thm:clt-lqr} yields
\begin{align}
\label{eq:lqr:variance}
\sigma_{t, \mathrm{asym}}^2(x_t)  = x_t^\top S_t x_t +   v_t,
\end{align}
where
$S_t$, and $v_t$ are defined in \Cref{thm:clt-lqr}.
Moreover, from \eqref{eq:lqr:prop-curr-variance}, we recall the
\emph{propagated future variance} $\sigmaprop^2(x_t)$ 
and the \emph{current stage variance} $\sigmacurr^2(x_t)$.

\Cref{fig:lqr:hist} depicts histograms of the random variable
  $N^{1/2}(\hat{V}_{t,N}(x_t) - V_{t}(x_t))$. The distributions are
  shown for a fixed state using $N=1000$ samples and are
  generated from $10,000$ independent trials. The subfigures
  present the histograms at three distinct time steps.
The variance grows for earlier
time steps (from $t=20$ down to $t=1$), reflecting the accumulation of
stochastic error from future stages.

\begin{figure}[t]
  \centering
  \begin{subfigure}[b]{0.32\linewidth}
    \includegraphics[width=0.95\linewidth]{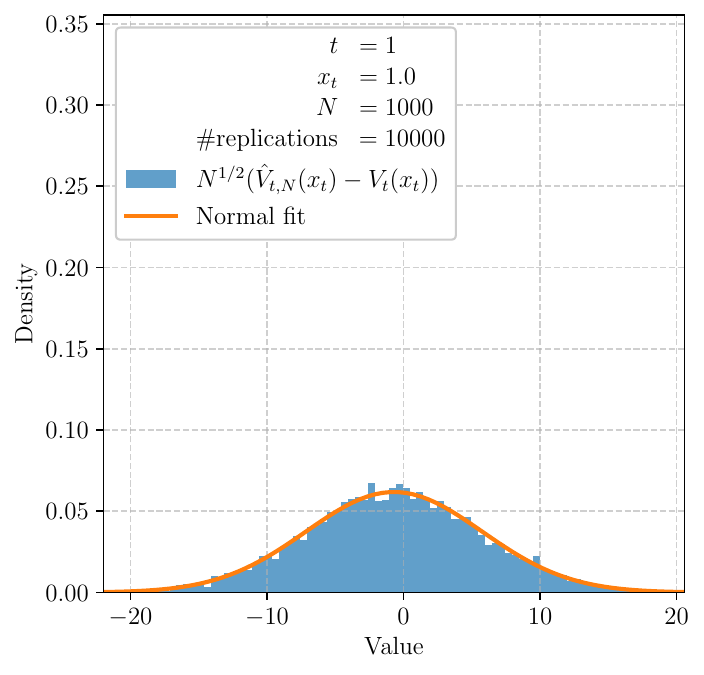}
    \caption{$t=1$.}
  \end{subfigure}
  \hfill
  \begin{subfigure}[b]{0.32\linewidth}
    \includegraphics[width=0.95\linewidth]{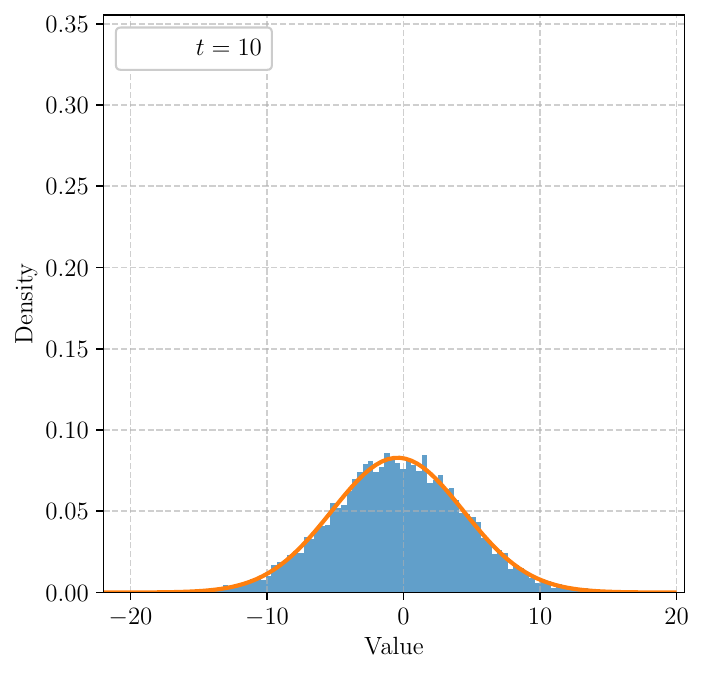}
    \caption{$t=10$.}
  \end{subfigure}
  \hfill
  \begin{subfigure}[b]{0.32\linewidth}
    \includegraphics[width=0.95\linewidth]{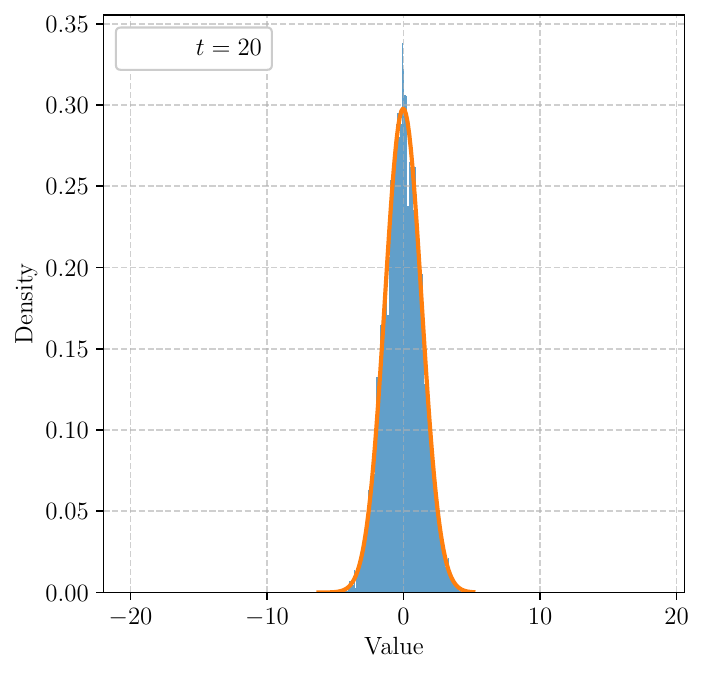}
    \caption{$t=20$.}
  \end{subfigure}
  \caption{%
Empirical distributions of the normalized value function estimation
error at different time stages
for the linear quadratic control problem.
The histograms visualize the empirical
distribution
of the scaled error, $N^{1/2}(\hat{V}_{t,N}(x_t) - V_{t}(x_t))$.
These results are based on $10,000$
independent replications, each using a sample size of $N=1000$ for the
state $x_t=1$. Each subfigure includes a curve representing a fitted normal distribution with the mean and standard deviation estimated from the simulation data.
}
\label{fig:lqr:hist}
\end{figure}

\Cref{fig:lqr:probplot} provides normal probability plots of the
normalized error,   $N^{1/2}(\hat{V}_{t,N}(x_t) - V_{t}(x_t))$.

\begin{figure}[t]
  \centering
\begin{subfigure}[b]{0.32\linewidth}
  \includegraphics[width=0.95\linewidth]{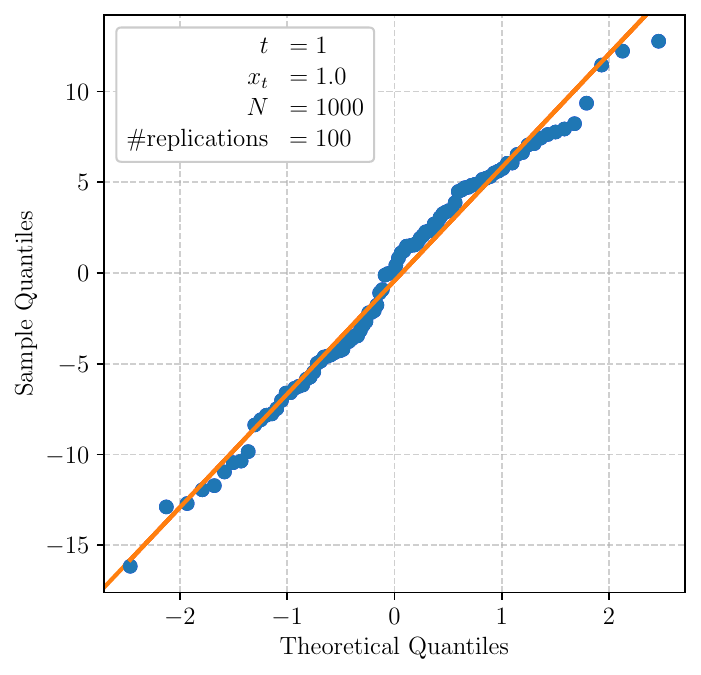}
  \caption{Time period $t=1$.}
\end{subfigure}
\hfill
\begin{subfigure}[b]{0.32\linewidth}
  \includegraphics[width=0.95\linewidth]{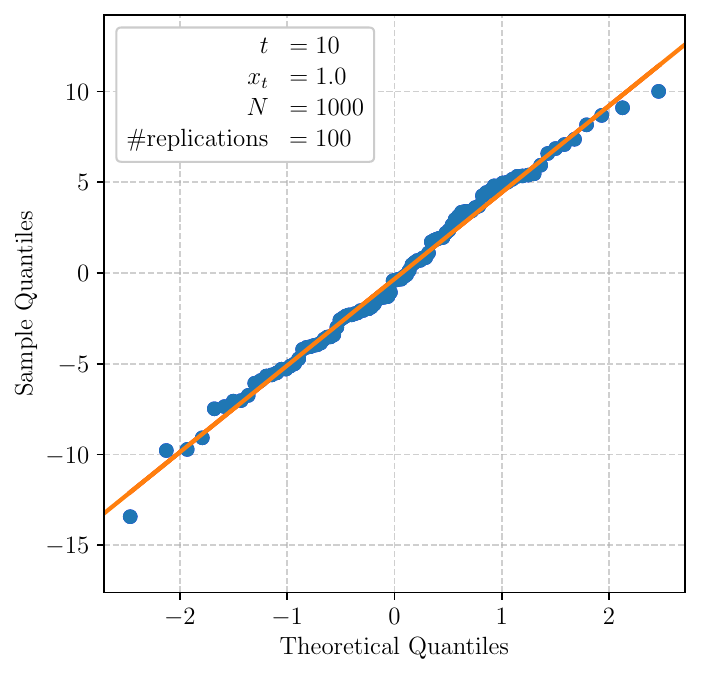}
  \caption{Time period $t=10$.}
\end{subfigure}
\hfill
\begin{subfigure}[b]{0.32\linewidth}
  \includegraphics[width=0.95\linewidth]{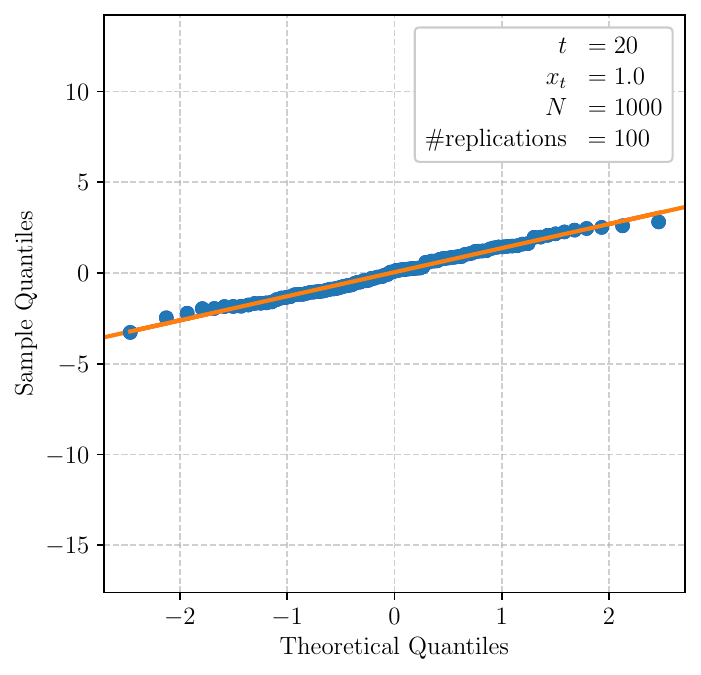}
  \caption{Time period $t=20$.}
\end{subfigure}

  \caption{%
Normal probability plots of the normalized value function estimation error,
$N^{1/2}(\hat{V}_{t,N}(x_t) - V_t(x_t))$,
for  time stages $t \in \{1, 10, 20\}$.
Each plot illustrates the empirical distribution of the scaled estimation error for a fixed state $x_t = 10$, sample size $N = 1000$,
and number of  replications equal to $100$.
}
\label{fig:lqr:probplot}
\end{figure}

\Cref{fig:lqr:variances-1} illustrates the quadratic variance
  function, $\sigma_{t, \mathrm{asym}}^2(x_t)$, over time.
The left panel shows the evolution of the parameters $S_t$
    (quadratic term) and $v_t$ (constant term) across the time horizon.
The right panel displays snapshots of the complete variance
    function $\sigma_{t, \mathrm{asym}}^2$ as a function of the state $x_t$ at
    different time points.
The variance is largest at early time steps and decreases as
time $t$
approaches the terminal time, showing how uncertainty compounds backward
from the future.

\begin{figure}[t]
  \centering
  \begin{subfigure}[b]{0.40\linewidth}
    \includegraphics[width=0.95\linewidth]{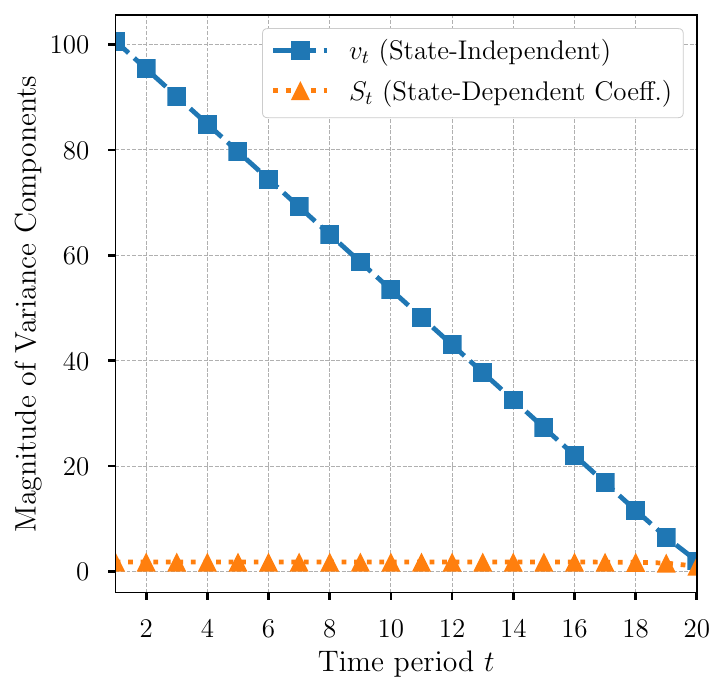}
    \caption{Quadratic coefficient
$S_t$ and the state-independent offset $v_t$ over the time horizon.}
  \end{subfigure}
  \hfill
  \begin{subfigure}[b]{0.40\linewidth}
    \includegraphics[width=0.95\linewidth]{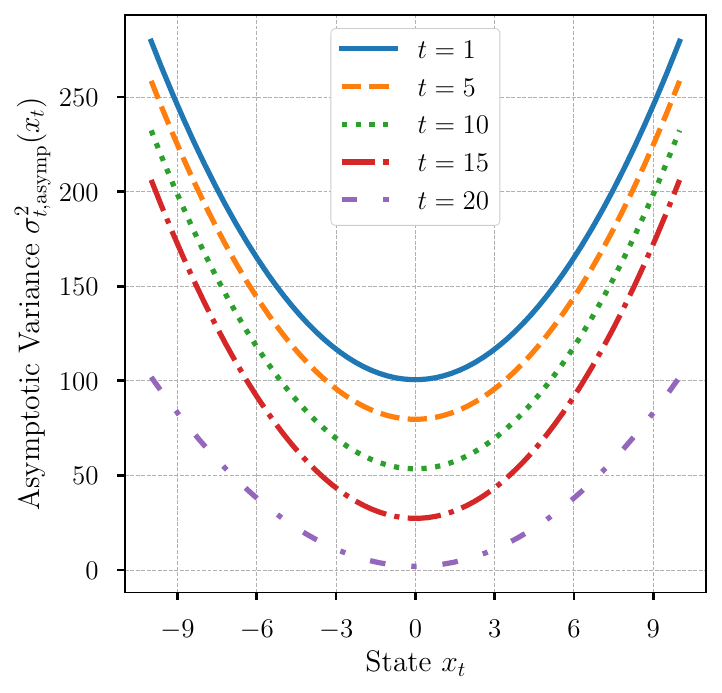}
    \caption{%
Snapshots of the variance function $\sigma_{t, \mathrm{asym}}^2(\cdot)$
at different time points.}
  \end{subfigure}
  \caption{%
Visualization of the quadratic asymptotic variance function.
The asymptotic variance is given by $\sigma_{t, \mathrm{asym}}^2(x_t) = x_t^\top S_t x_t + v_t$
per \eqref{eq:lqr:variance}.
}
\label{fig:lqr:variances-1}
\end{figure}

\Cref{fig:lqr:variances-2} shows the decomposition of the total
  variance into its two constituent parts: the propagated future
  variance, $\sigmaprop^2(x_t)$ and the current stage variance,
  $\sigmacurr^2(x_t)$, as provided in \eqref{eq:lqr:prop-curr-variance}.
The plots show these two components and their sum
  (the total variance $\sigma_{t, \mathrm{asym}}^2$) over the time horizon, evaluated
  at two distinct states: $x_t = 1/2$ (left panel) and
  $x_t = 3/2$ (right panel).
Here, propagated future
variance, $\sigmaprop^2$, seems to be the dominant contributor to the total variance. The stage variances shown in \Cref{fig:lqr:variances-1,fig:lqr:variances-2} remain constant over most time periods, which we attribute to the simple structure of the LQR problem.

\begin{figure}[t]
  \centering
  \begin{subfigure}[b]{0.40\linewidth}
    \includegraphics[width=0.95\linewidth]{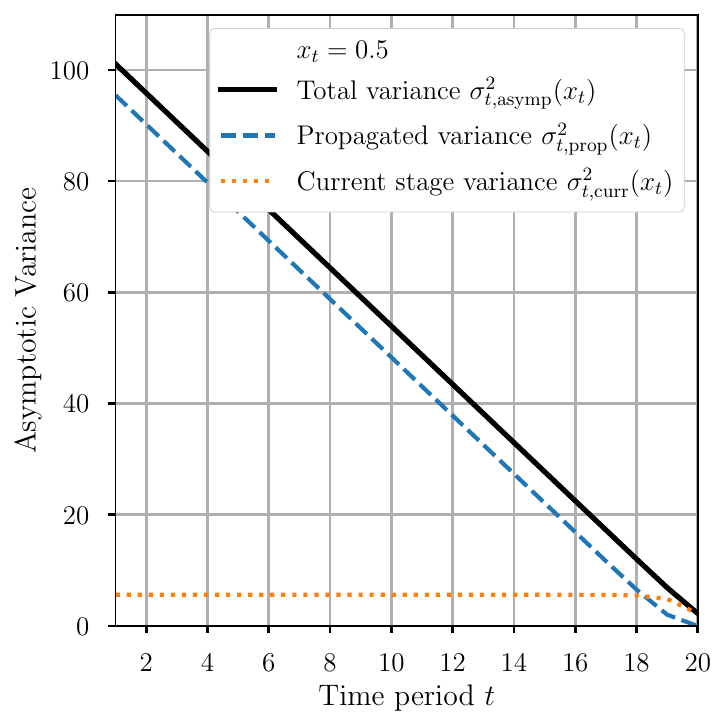}
    \caption{Variances  evaluated at $x_t = 1/2$.}
  \end{subfigure}
  \hfill
  \begin{subfigure}[b]{0.40\linewidth}
    \includegraphics[width=0.95\linewidth]{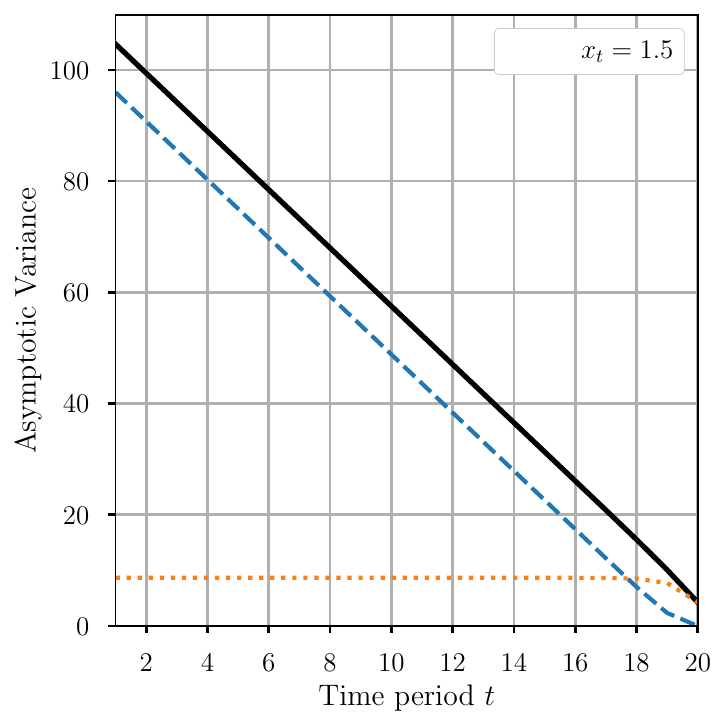}
    \caption{Variances evaluated at $x_t = 3/2$.}
  \end{subfigure}
  \caption{%
Decomposition of the total asymptotic variance $\sigma_{t, \mathrm{asym}}^2(x_t)$ into its two
sources: the \textit{propagated future variance}, $\sigmaprop^2(x_t)$
(see \eqref{eq:lqr:prop-curr-variance}), which is inherited from future
stages, and the \textit{current stage variance}, $\sigmacurr^2(x_t)$
(see \eqref{eq:lqr:prop-curr-variance}), which is generated locally. The
panels plot these two components and their sum
over the time horizon for two distinct states.
}
\label{fig:lqr:variances-2}
\end{figure}

\section{Summary and conclusions}
\label{sec-summary}

As  mentioned in the introduction section, the asymptotics of the SAA method are well understood for static (one-stage) stochastic programs. On the other hand, there are virtually no results for CLT-type asymptotics of  the SAA method applied to SOC or multistage stochastic programming. This paper takes the first  step toward addressing this gap.
Our main result, the CLT in \Cref{thm:clt}, relies on the stochastic
equicontinuity of the SAA dynamic programming operators. We use
random empirical process theory to provide sufficient conditions for this
key requirement. Because these operators are functions of value functions,
we introduce \Cref{ass-4}, which ensures that the classes of SAA optimal
value functions have sufficiently small covering numbers.

This work opens several avenues for future research. Open questions
include developing limit theorems for non-iid, state- and
control-dependent probability distributions and identifying broader problem
classes that satisfy the stochastic equicontinuity-type condition.
In particular, SOC with Markovian noise could be considered; see, e.g.,
\cite{Song1993,Milz2026}.
As illustrated in \Cref{sec:lqr-numerical} for the LQR example, the numerical
results suggest that, for fixed sample size $N$, the asymptotic
variance of the scaled SAA value
$N^{1/2}\big(\hat V_{t,N}(x_t)-V_t(x_t)\big)$
increases approximately linearly as one moves backward in time.
In the same example, the true value also grows linearly with the horizon.
This suggests that, in such settings, the relative statistical error of
the SAA approach, measured by the ratio of the standard deviation to the
mean, may behave like $O(N^{-1/2}T^{-1/2})$. Identifying conditions
under which such behavior holds more generally remains an interesting
question for future research.

\paragraph{Acknowledgments}
We thank Xin Chen for many fruitful discussions on this topic, for
suggesting that we consider convex value functions, and for pointing us
to relevant references. We are also grateful to Ashwin Pananjady for
discussions and references related to functional limit theorems, and to
Olena Melnikov for helpful comments on an earlier draft.

Some derivations in \Cref{sec:lqr-proofs} were initially drafted with assistance
from Google's Gemini 3.1 Pro, 
and ChatGPT 5.4 Pro, based on
theorems and equations provided by the authors. The model assisted in expanding
recursions, computing intermediate steps, and generating \LaTeX\ code. All
results were subsequently revised and verified by the authors.
In addition, Gemini 3.1 Pro and ChatGPT 5.4 Pro were used to produce initial drafts of selected paragraphs, based on bullet points provided by the authors.

\paragraph{Reproducibility of computational results}

Computer code that allows the reader to reproduce the computational results in this manuscript is available at
\url{https://doi.org/10.5281/zenodo.16733442}.

\appendix 

\section{Proofs of \Cref{lem:lqr-assumptions,thm:clt-lqr}}
\label{sec:lqr-proofs}

\begin{proof}[Proof of \Cref{lem:lqr-assumptions}]
We verify the hypotheses of \Cref{thm:general-clt}.
As discussed, $\mathcal{C}$ is a separable Banach space, $\mathcal{D}_t$ is closed, and $f_{T+1} \in \mathcal{D}_{T+1}$. The set $\mathcal{D}_t$ is affine-isomorphic to $\mathbb{R}^n \times \mathbb{R}$,
which we use repeatedly throughout the proof.
We prepare the verification of each part of  \Cref{lem:lqr-assumptions}.
For
$
W(x)=x^\top H_{t+1}x+b^\top x+c \in \mathcal D_{t+1},
$
direct computations give
(cf.\  \cite[Section 4.1]{Bertsekas2005})
\begin{align}
\label{eq:TW-lqr}
[\mathcal T_t W](x_t)
=
x_t^\top H_t x_t
+
b^\top M_t x_t
+
c
+
\mathbb E[\xi_t^\top H_{t+1}\xi_t]
-
\tfrac14\,b^\top B_t K_t^{-1} B_t^\top b.
\end{align}
With 
$
\bar\xi_{t,N} \coloneqq (1/N) \sum_{i=1}^N \xi_{ti}
$,
we also obtain
\begin{align}
\label{eq:hatTW-lqr}
\begin{aligned}
[\hat{\mathcal T}_{t,N} W](x_t)
&=
x_t^\top H_t x_t
+
b^\top M_t x_t
+
c
+
\tfrac1N \tsum_{i=1}^N \xi_{ti}^\top H_{t+1}\xi_{ti}
\\
&\quad
+
2\bar\xi_{t,N}^\top H_{t+1} M_t x_t
+
b^\top
\bigl(I-B_tK_t^{-1}B_t^\top H_{t+1}\bigr)\bar\xi_{t,N}
\\
&\quad
-
\bar\xi_{t,N}^\top
H_{t+1}B_tK_t^{-1}B_t^\top H_{t+1}\bar\xi_{t,N}
-
\tfrac14\,b^\top B_t K_t^{-1} B_t^\top b.
\end{aligned}
\end{align}

(i) We show that $\mathcal{T}_t$ is Hadamard directionally differentiable at $V_{t+1}$ tangentially to $\mathcal{D}_{t+1}$.
Let $\tau_N \downarrow 0$ and $W_N \in \mathcal{C}$ such that $V_{t+1} + \tau_N W_N \in \mathcal{D}_{t+1}$.
Since 
$V_{t+1} \in \mathcal{D}_{t+1}$,
$\nabla^2 W_N(0) = 0$, giving $W_N(x_t) = b_N^\top x_t + c_N$ with $W_N \to W$, $b_N \to b$, and $c_N \to c$.
Using \eqref{eq:TW-lqr}, we have
\begin{align*}
[\mathcal{T}_t(V_{t+1} + \tau_N W_N)](x_t) 
&= x_t^\top H_t x_t + q_t \\
&\quad + \tau_N (b_N^\top M_t x_t + c_N) \\
&\quad - \tfrac{\tau_N^2}{4} b_N^\top B_t (R_t + B_t^\top H_{t+1} B_t)^{-1} B_t^\top b_N.
\end{align*}
Since $\mathcal{T}_t V_{t+1}(x_t) = x_t^\top H_t x_t + q_t$, we obtain $[\mathcal{T}_t' (V_{t+1}; W)](x_t) = b^\top M_t x_t + c$.
The inclusion $\mathcal{T}_t \mathcal{D}_{t+1} \subset \mathcal{D}_t$ follows from \eqref{eq:TW-lqr}.

(ii) The operators $\hat{\mathcal{T}}_{t,N}$ are continuous by definition and independent by sample independence. The inclusion $\hat{\mathcal{T}}_{t,N} \mathcal{D}_{t+1} \subset \mathcal{D}_t$ follows from \eqref{eq:hatTW-lqr}.

(iii) Using \eqref{eq:TW-lqr} and \eqref{eq:hatTW-lqr}, we have
\begin{align*}
\big[ \hat{\mathcal{T}}_{t,N}(V_{t+1}) - \mathcal{T}_t(V_{t+1}) \big](x_t) 
&= \Big( \tfrac{1}{N} \tsum_{i=1}^N \xi_{ti}^\top H_{t+1} \xi_{ti} - \mathbb{E}[\xi_t^\top H_{t+1} \xi_t] \Big) \\
&\quad + 2 \bar{\xi}_{t,N}^\top H_{t+1} M_t x_t \\
&\quad - \bar{\xi}_{t,N}^\top H_{t+1} B_t (R_t + B_t^\top H_{t+1} B_t)^{-1} B_t^\top H_{t+1} \bar{\xi}_{t,N}.
\end{align*}
The mapping $\iota_t \colon \mathbb{R}^n \times \mathbb{R} \to C_w(\mathbb{R}^n)$ defined by $[\iota_t(b, v)](x_t) \coloneqq b^\top M_t x_t + v$ is continuous. Applying the multivariate CLT,  and using the continuous mapping theorem, and Slutsky's theorem,
\begin{align*}
N^{1/2} \big[ \hat{\mathcal{T}}_{t,N}(V_{t+1}) - \mathcal{T}_t(V_{t+1}) \big]
\dst \mathfrak{H}_t(\cdot) \coloneqq 2 \vartheta_t^\top H_{t+1} M_t \cdot + \zeta_t
\quad \text{in} \quad C_w(\mathbb{R}^n).
\end{align*}

(iv) We verify the stochastic equicontinuity-type condition. 
Since $\hat V_{t+1,N},V_{t+1}\in \mathcal D_{t+1}$, there exist
random vectors $\hat b_N\in \mathbb R^n$ and random scalars
$\hat c_N\in \mathbb R$ such that
\[
\hat V_{t+1,N}(x)=V_{t+1}(x)+\hat b_N^\top x+\hat c_N.
\]
Because
$
N^{1/2}(\hat V_{t+1,N}-V_{t+1}) \dst \mathfrak G_{t+1}
$
in $C_w(\mathbb R^n)$
and $\mathcal D_{t+1}$ is affine-isomorphic to
$\mathbb R^n\times \mathbb R$, we have
\[
\hat b_N=O_p(N^{-1/2}),
\quad 
\text{and}
\quad 
\hat c_N=O_p(N^{-1/2}).
\]

Now, let
$
W(x)=x^\top H_{t+1}x+b^\top x+c \in \mathcal D_{t+1}
$.
Subtracting \eqref{eq:TW-lqr} from \eqref{eq:hatTW-lqr} yields
\begin{align*}
[\hat{\mathcal T}_{t,N}W-\mathcal T_tW](x_t)
&=
\Big(
\tfrac1N\tsum_{i=1}^N \xi_{ti}^\top H_{t+1}\xi_{ti}
-
\mathbb E[\xi_t^\top H_{t+1}\xi_t]
\Big)
\\
&\quad
+
2\bar\xi_{t,N}^\top H_{t+1}M_t x_t
+
b^\top
\bigl(I-B_tK_t^{-1}B_t^\top H_{t+1}\bigr)\bar\xi_{t,N}
\\
&\quad
-
\bar\xi_{t,N}^\top
H_{t+1}B_tK_t^{-1}B_t^\top H_{t+1}\bar\xi_{t,N}.
\end{align*}
Evaluating this identity at
$\hat V_{t+1,N}=V_{t+1}+\hat b_N^\top x+\hat c_N$
and at $V_{t+1}$, and subtracting, yields
\begin{align*}
\big(
[\hat{\mathcal T}_{t,N}-\mathcal T_t](\hat V_{t+1,N})
-
[\hat{\mathcal T}_{t,N}-\mathcal T_t](V_{t+1})
\big)(x_t)
&=
\hat b_N^\top
\bigl(I-B_tK_t^{-1}B_t^\top H_{t+1}\bigr)\bar\xi_{t,N},
\end{align*}
which is state-independent. Since
$\bar\xi_{t,N}=O_p(N^{-1/2})$, it follows that
\[
\|
[\hat{\mathcal T}_{t,N}-\mathcal T_t](\hat V_{t+1,N})
-
[\hat{\mathcal T}_{t,N}-\mathcal T_t](V_{t+1})
\|_{C_w(\mathbb R^n)}
=
O_p(N^{-1})
=
o_p(N^{-1/2}),
\]
which proves the stochastic equicontinuity-type condition.
\end{proof}

\begin{proof}[Proof of \Cref{thm:clt-lqr}]
By \Cref{lem:lqr-assumptions}, the hypotheses of \Cref{thm:general-clt} hold. 
We proceed by backward induction.
If $N^{1/2}(\hat{V}_{t+1,N} - V_{t+1}) \dst \mathfrak{G}_{t+1}$,
then
$N^{1/2}(\hat{V}_{t,N} - V_t) \dst \mathfrak{G}_t$ with
\begin{align*}
\mathfrak{G}_{t}(x_t) &= \mathbb{E}_{\xi_t \sim P_t}[\mathfrak{G}_{t+1}(M_t x_t + \xi_t)] + \mathfrak{H}_t(x_t).
\end{align*}

For $t=T$, $\mathfrak{G}_{T+1} = 0$, giving $\mathfrak{G}_T(x_T) = \mathfrak{H}_T(x_T) = b_T^\top x_T + c_T$ with $b_T = 2 M_T^\top H_{T+1} \vartheta_T$ and $c_T = \zeta_T$.
Assume $\mathfrak{G}_{t+1}(x_{t+1}) = b_{t+1}^\top x_{t+1} + c_{t+1}$. Since $\mathbb{E}[\xi_t] = 0$,
\begin{align}
\label{eq:Ap11524}
\mathbb{E}_{\xi_t \sim P_t}[\mathfrak{G}_{t+1}(M_t x_t + \xi_t)] 
&= b_{t+1}^\top M_t x_t + c_{t+1}.
\end{align}
Substituting this and $\mathfrak{H}_t(x_t) = 2 \vartheta_t^\top H_{t+1} M_t x_t + \zeta_t$ into the recursion,
\begin{align*}
\mathfrak{G}_{t}(x_t)
&= \big(2 \vartheta_t^\top H_{t+1} M_t x_t + \zeta_t\big) + \big(b_{t+1}^\top M_t x_t + c_{t+1}\big) \\
&= \big(M_t^\top(2 H_{t+1} \vartheta_t + b_{t+1})\big)^\top x_t + \big(\zeta_t + c_{t+1}\big).
\end{align*}
Thus, $\mathfrak{G}_t(x_t) = b_t^\top x_t + c_t$ with $b_{t} \coloneqq M_{t}^\top(2H_{t+1}\vartheta_t + b_{t+1})$ and $c_t \coloneqq \zeta_t + c_{t+1}$.

We establish the moment recursions. 
We recall $S_t =\var(b_t)$, $v_t = \var(c_t)$, 
and $s_t=\cov(b_t, c_t)$.
The current stage noises $\vartheta_t, \zeta_t$ are independent of $b_{t+1}, c_{t+1}$.
Hence
\begin{align*}
S_t &= \var\big(M_t^\top(2 H_{t+1} \vartheta_t + b_{t+1})\big) 
    = M_t^\top \big(\var(b_{t+1}) + 4 H_{t+1} \var(\vartheta_t) H_{t+1}\big) M_t \\
    &= M_t^\top(S_{t+1} + 4 H_{t+1} \Sigma_t H_{t+1}) M_t.
\end{align*}
For the covariance vector $s_t$,
\begin{align*}
s_t &= \cov\big(M_t^\top(2 H_{t+1} \vartheta_t + b_{t+1}), \zeta_t + c_{t+1}\big) 
    = M_t^\top \big(\cov(b_{t+1}, c_{t+1}) + 2 H_{t+1} \cov(\vartheta_t, \zeta_t)\big) \\
    &= M_t^\top (s_{t+1} + \gamma_t),
\end{align*}
where $\gamma_t = 2 H_{t+1} \mathbb{E}[(\xi_t^\top H_{t+1} \xi_t) \xi_t]$. 
For the scalar variance $v_t$,
\begin{align*}
v_t &= \var(\zeta_t + c_{t+1}) = v_{t+1} + \var(\xi_t^\top H_{t+1} \xi_t).
\end{align*}

The total asymptotic variance decomposes as
\begin{align*}
\sigma_{t, \mathrm{asym}}^2(x_t) = \var(b_t^\top x_t + c_t) = x_t^\top S_t x_t + 2 x_t^\top s_t + v_t.
\end{align*}
For the propagated variance $\sigmaprop^2(x_t)$,
we use \eqref{eq:Ap11524}  to obtain 
\begin{align*}
\sigmaprop^2(x_t) & = 
\var\big(\mathbb{E}_{\xi_t \sim P_t}[\mathfrak{G}_{t+1}(M_t x_t + \xi_t)]\big) 
= \var(b_{t+1}^\top M_t x_t + c_{t+1}) 
\\
& = x_t^\top M_t^\top S_{t+1} M_t x_t + 2 x_t^\top M_t^\top s_{t+1} + v_{t+1}.
\end{align*}
For the current stage variance $\sigmacurr^2(x_t)$, the optimal action $\pi_t^*$ fixes deterministic components of $f_t$, yielding
\begin{align*}
\sigmacurr^2(x_t) &= \var_{\xi_t \sim P_t}\big[V_{t+1}(M_t x_t + \xi_t)\big].
\end{align*}
Expanding $V_{t+1}(M_t x_t + \xi_t) = (M_t x_t + \xi_t)^\top H_{t+1} (M_t x_t + \xi_t) + q_{t+1}$ and isolating stochastic terms,
\begin{align*}
\sigmacurr^2(x_t) &= \var_{\xi_t}\big[2x_t^\top M_t^\top H_{t+1} \xi_t + \xi_t^\top H_{t+1} \xi_t\big] \\
&= \var(2x_t^\top M_t^\top H_{t+1} \xi_t) + \var(\xi_t^\top H_{t+1} \xi_t) + 2 \cov(2x_t^\top M_t^\top H_{t+1} \xi_t, \xi_t^\top H_{t+1} \xi_t) \\
&= 4 x_t^\top M_t^\top H_{t+1} \Sigma_t H_{t+1} M_t x_t + \var(\xi_t^\top H_{t+1} \xi_t) + 2 x_t^\top M_t^\top \gamma_t.
\end{align*}
\end{proof}

\section{Statistical limit theorems for SOC
under nonunique optimal policies}
\label{app-B}

We now extend the analysis from \Cref{sec-cltsaa} to the case of nonunique optimal policies.
In this setting, the derivative of the Bellman minimization operator is
generally nonlinear, so the limiting law of the SAA value functions need
not be Gaussian;  for the terminal decision stage $t=T$ this was already pointed out in Remark \ref{rem-nonunique}.  Nevertheless, by working with the action-value
operators before minimization, we obtain a backward recursion that
characterizes the weak limits stage by stage.
Our main result of this section, Theorem \ref{thm:nonunique-limits}, can be viewed as a corollary of \Cref{thm:general-clt}.
We recall the definition of the action-value operators
$\mathcal{Q}_t$, $\hat{\mathcal Q}_{t,N} \colon C(\mathcal{X}_{t+1}) \to C(\mathcal{X}_t  \times \mathcal{U}_t)$ from \eqref{eq:q-function}
and \eqref{eq:hatq-function}, respectively,
and  the definition of $\Phi_t$ from \eqref{Phi-t}.
For $\phi \in C(\mathcal{X}_t \times \mathcal{U}_t)$,    we recall 
from \Cref{lem:Gt-Hadamard}, the optimal solution set
\[\sU^*_t(\phi; x_t) = \argmin_{u_t \in \mathcal{U}_t} \phi(x_t, u_t).\]

If \Cref{ass-1,ass-1',ass-2}(i)--(iii) hold, then 
adapting the proof of \Cref{pr-1} yields
\[N^{1/2}(\hat{V}_{T,N} - V_{T}) \dst \mathfrak{G}_{T},\]
where $\mathfrak{G}_{T}$ is a potentially non-Gaussian random element with values in $C(\mathcal{X}_{T})$ given by
\begin{equation}
\label{gausselement}
\mathfrak{G}_T(x_T) = \inf_{u_T \in \sU^*_T\big (\mathcal{Q}_T(f_{T+1}); x_T\big)} \mathfrak{Y}_T(x_T, u_T),
\end{equation}
and $\mathfrak{Y}_T$ 
is a mean-zero Gaussian process in $C(\mathcal{X}_T \times \mathcal{U}_T)$ given by
\begin{align*}
N^{1/2} \big[ \hat{\mathcal{Q}}_{T,N}(f_{T+1}) - \mathcal{Q}_T(f_{T+1}) \big] \dst \mathfrak{Y}_T.
\end{align*}
 This is equivalent to the formulation \eqref{asymvalue} of the respective asymptotics, with $\sU^*_T\big (\mathcal{Q}_T(f_{T+1}); x_T\big)$ coinciding with 
 $\U_T^*(x_T)$ defined  in \eqref{asymval-2}.

\begin{theorem}
\label{thm:nonunique-limits}
Let \Cref{ass-1,ass-1',ass-2}\textup{(i)--(iii)} hold.
Let  $t \in \{1, \ldots, T\}$.
Suppose that $N^{1/2}(\hat{V}_{t+1,N} - V_{t+1}) \dst \mathfrak{G}_{t+1}$,
where $\mathfrak{G}_{t+1}$ is a random element in $C(\mathcal{X}_{t+1})$ 
\emph{(induction hypothesis)}.
Then 
\begin{align}
\label{eq:Apr12026}
N^{1/2} \big[ \hat{\mathcal{Q}}_{t,N}(V_{t+1}) - \mathcal{Q}_t(V_{t+1}) \big] \dst \mathfrak{Y}_t,
\end{align}
    where $\mathfrak{Y}_t$ is a mean-zero Gaussian process in $C(\mathcal{X}_t \times \mathcal{U}_t)$, 
with covariance function
\begin{align*}
\cov(\mathfrak{Y}_t(x_t, u_t), \mathfrak{Y}_t(x_t^\prime, u_t^\prime)) = \cov\big(\Phi_t(x_t, u_t, \xi_t), \Phi_t(x_t^\prime, u_t^\prime, \xi_t)\big).
\end{align*}
Suppose further that  the stochastic equicontinuity-type condition \eqref{eq:q-equicontinuity}
holds.
Then 
\begin{align}
\label{eq:joint-q-limit}
N^{1/2} \big[ \hat{\mathcal{Q}}_{t,N}(\hat{V}_{t+1,N}) - \mathcal{Q}_t(V_{t+1}) \big] \dst \mathfrak{Z}_t,
\end{align}
where $\mathfrak{Z}_t$ is a random element in $C(\mathcal{X}_t \times \mathcal{U}_t)$ given 
by
\[\mathfrak{Z}_t(x_t, u_t) \coloneqq \mathbb{E}_{\xi_t \sim P_t} [ \mathfrak{G}_{t+1}(F_t(x_t, u_t, \xi_t)) ] + \mathfrak{Y}_t(x_t, u_t).\]
Moreover, 
$N^{1/2}(\hat{V}_{t,N} - V_{t})$ converges in distribution to a random element $\mathfrak{G}_{t}$ in $C(\mathcal{X}_{t})$ given by
\begin{align}
\label{eq:nonunique-limits}
\mathfrak{G}_{t}(x_t) = \inf_{u_t \in \sU^*_t\big (\mathcal{Q}_{t}(V_{t+1}); x_t\big )} \big\{\,  \mathbb{E}_{\xi_t \sim P_t} [ \mathfrak{G}_{t+1}(F_t(x_t, u_t, \xi_t)) ] + \mathfrak{Y}_t(x_t, u_t) 
\, 
\big\}.
\end{align}
\end{theorem}

Before establishing \Cref{thm:nonunique-limits}, let us comment on the limit law recursion in 
\eqref{eq:nonunique-limits}.

\begin{remark}
\normalfont 
The limit processes $\mathfrak{G}_t$ in \eqref{eq:nonunique-limits} satisfy a dynamic programming principle.
Indeed, let $\mathfrak{T}_t \colon C(\mathcal{X}_{t+1}) \to C(\mathcal{X}_t)$ be defined by
\begin{equation}
\label{eq:nonunique-limit-operator}
[\mathfrak{T}_t(W)](x_t) \coloneqq  \inf_{u_t \in \sU^*_t\big(\mathcal{Q}_{t}(V_{t+1}); x_t\big)} \big\{ \mathfrak{Y}_t(x_t, u_t) + \mathbb{E}_{\xi_t \sim P_t}[W(F_t(x_t, u_t, \xi_t))] \big\},
\end{equation}
where $\mathfrak{Y}_t$ is as in \eqref{eq:Apr12026}. Then, the limit processes satisfy the backward recursion
\begin{align*}
\mathfrak{G}_t  = \mathfrak{T}_t \mathfrak{G}_{t+1}, \quad t = 1, \ldots, T-1,
\end{align*}
with the terminal condition at stage $T$ given by \eqref{gausselement}.
$\hfill \square$
\end{remark}

\begin{proof}[{Proof of \Cref{thm:nonunique-limits}}]
We verify the hypotheses of \Cref{thm:general-clt} 
applied with $\mathcal{B}_t \coloneqq \mathcal{Q}_t$
and $\hat{\mathcal{B}}_{t,N} \coloneqq  \hat{\mathcal{Q}}_{t,N}$ to demonstrate
\eqref{eq:joint-q-limit}. 
Subsequently, we use the Delta Theorem to show \eqref{eq:nonunique-limits}.

According to \Cref{lem:composition-is-hadamard}, 
 $\mathcal{Q}_t$ is Hadamard directionally differentiable at $V_{t+1}$. 
Combined with  $N^{1/2}(\hat{V}_{t+1,N} - V_{t+1}) \dst \mathfrak{G}_{t+1}$, the
Delta Theorem ensures
\begin{align}
\label{eq:Apr11110pm}
N^{1/2} \big[ \mathcal{Q}_t(\hat{V}_{t+1,N}) - \mathcal{Q}_t(V_{t+1}) \big]\dst \mathbb{E}_{\xi_t \sim P_t} [ \mathfrak{G}_{t+1}(F_t(\cdot, \cdot, \xi_t)) ].
\end{align}
Moreover, standard uniform CLTs ensure \eqref{eq:Apr12026}, 
which corresponds to \eqref{eq:general-intermediate-limit} in \Cref{thm:general-clt}.
Therefore, an application of \Cref{thm:general-clt} yields \eqref{eq:joint-q-limit}.

Now, we recall the definition of the infimum $G_t$ operator from \Cref{lem:Gt-Hadamard},
which is Hadamard directionally differentiable at
$\mathcal{Q}_t(V_{t+1})$, with the directional derivative given by (compare with \eqref{dirderiv})
\begin{equation*}
[G_t^\prime(\mathcal{Q}_{t}(V_{t+1}); \eta)](x_t) = \inf_{u_t \in \sU^*_t\big(\mathcal{Q}_{t}(V_{t+1}); x_t\big)} \eta(x_t, u_t).
\end{equation*}
Applying the functional Delta Theorem (see, e.g., \cite[Theorem 9.74]{SDR}), 
and the statistical limit in \eqref{eq:joint-q-limit},
\begin{align*}
N^{1/2}(\hat{V}_{t,N} - V_t) = N^{1/2} \big[ G_t\big(\hat{\mathcal{Q}}_{t,N}(\hat{V}_{t+1,N})\big) - G_t\big(\mathcal{Q}_t(V_{t+1})\big) \big] 
\dst G_t^\prime(\mathcal{Q}_{t}(V_{t+1}); \mathfrak{Z}_t).
\end{align*}
Substituting the definition of $\mathfrak{Z}_t$ yields the limit process $\mathfrak{G}_t$ defined in \eqref{eq:nonunique-limits}.
\end{proof}

\section{Proof of \Cref{lem:equicontinuity-equivalence}}
\label{sect:appendix-Random-empirical-processes}

\begin{proof}[{Proof of \Cref{lem:equicontinuity-equivalence}}]
Let $G_t \colon C(\mathcal{X}_t \times \mathcal{U}_t) \to C(\mathcal{X}_t)$ be the infimum operator defined by $[G_t(\phi)](x_t) \coloneqq \inf_{u_t \in \mathcal{U}_t} \phi(x_t, u_t)$ from 
\Cref{lem:Gt-Hadamard}. 
Using the definitions of the action-value functions from \eqref{eq:q-function}
and \eqref{eq:hatq-function},
and those of the dynamic programming operators \eqref{saaderiv-1} and
\eqref{saaderiv-2}, 
we have
 $\mathcal{T}_t = G_t \circ \mathcal{Q}_t$ and $\hat{\mathcal{T}}_{t,N} = G_t \circ \hat{\mathcal{Q}}_{t,N}$. 
In particular, 
\begin{align}
\label{eq:T-Q-expansion}
\begin{aligned}
\big[\hat{\mathcal{T}}_{t,N}
-
\mathcal{T}_{t}\big](\hat{V}_{t+1,N})
-
\big[
\hat{\mathcal{T}}_{t,N}- \mathcal{T}_t
\big]
(V_{t+1})
& =
G_t\big(\hat{\mathcal{Q}}_{t,N}(\hat{V}_{t+1,N})\big) - G_t\big(\mathcal{Q}_t(V_{t+1})\big)
\\
& \quad 
-
\big[G_t\big(\mathcal{Q}_t(\hat{V}_{t+1,N})\big) - G_t\big(\mathcal{Q}_t(V_{t+1})\big) \big]
\\
& \quad 
 - \big[G_t\big(\hat{\mathcal{Q}}_{t,N}(V_{t+1})\big) - G_t\big(\mathcal{Q}_t(V_{t+1})\big)\big].
\end{aligned}
\end{align}

By \Cref{ass-2}(iv), the set of minimizers $\sU^*_t(\mathcal{Q}_t(V_{t+1}); x_t)$ is a singleton $\{\pi^*_t(x_t)\}$. Combined with \Cref{lem:Gt-Hadamard}, 
we find that $G_t$ is Hadamard  differentiable at  $\mathcal{Q}_t(V_{t+1})$ with $[G_t'(\mathcal{Q}_t(V_{t+1}); \eta)](x_t) = \eta(x_t, \pi^*_t(x_t))$
for all $\eta \in C(\mathcal{X}_t \times \mathcal{U}_t)$.

\Cref{thm:nonunique-limits} ensures
 $N^{1/2}(\hat{\mathcal{Q}}_{t,N}(\hat{V}_{t+1,N}) - \mathcal{Q}_t(V_{t+1})) \dst \mathfrak{Z}_t$,
where $\mathfrak{Z}_t$ is given in \eqref{eq:joint-q-limit}. 
Moreover, 
 $N^{1/2}(\mathcal{Q}_t(\hat{V}_{t+1,N}) - \mathcal{Q}_t(V_{t+1}))
$
converges in distribution as shown in \eqref{eq:Apr11110pm}, 
and \eqref{eq:Apr12026} ensures the convergence in distribution of
 $N^{1/2}(\hat{\mathcal{Q}}_{t,N}(V_{t+1}) - \mathcal{Q}_t(V_{t+1}))$.
Since $G_t$ is Hadamard directionally differentiable at $\mathcal{Q}_t(V_{t+1})$, 
the Delta Theorem (see \cite[Theorem 9.74]{SDR}) ensures the asymptotic expansions
in $C(\mathcal{X}_t)$, 
\[
\begin{aligned}
G_t\big(\hat{\mathcal{Q}}_{t,N}(\hat{V}_{t+1,N})\big) - G_t\big(\mathcal{Q}_t(V_{t+1})\big) &= G_t'\big(\mathcal{Q}_t(V_{t+1}); \hat{\mathcal{Q}}_{t,N}(\hat{V}_{t+1,N}) - \mathcal{Q}_t(V_{t+1})\big) + o_p(N^{-1/2}), \\
G_t\big(\mathcal{Q}_t(\hat{V}_{t+1,N})\big) - G_t\big(\mathcal{Q}_t(V_{t+1})\big) &= G_t'\big(\mathcal{Q}_t(V_{t+1}); \mathcal{Q}_t(\hat{V}_{t+1,N}) - \mathcal{Q}_t(V_{t+1})\big) + o_p(N^{-1/2}), \\
G_t\big(\hat{\mathcal{Q}}_{t,N}(V_{t+1})\big) - G_t\big(\mathcal{Q}_t(V_{t+1})\big) &= G_t'\big(\mathcal{Q}_t(V_{t+1}); \hat{\mathcal{Q}}_{t,N}(V_{t+1}) - \mathcal{Q}_t(V_{t+1})\big) + o_p(N^{-1/2}).
\end{aligned}
\]
Combined with \eqref{eq:T-Q-expansion}, and the linearity and boundedness of
$G_t'(\mathcal{Q}_t(V_{t+1}); \cdot)$, we obtain in $C(\mathcal{X}_t)$, 
\begin{align*}
\begin{aligned}
& \big[\hat{\mathcal{T}}_{t,N}
-
\mathcal{T}_{t}\big](\hat{V}_{t+1,N})
-
\big[
\hat{\mathcal{T}}_{t,N}- \mathcal{T}_t
\big]
(V_{t+1})
\\
& \quad = 
G_t'\big(\mathcal{Q}_t(V_{t+1}); \big[\hat{\mathcal{Q}}_{t,N} - \mathcal{Q}_t\big](\hat{V}_{t+1,N}) - \big[\hat{\mathcal{Q}}_{t,N} - \mathcal{Q}_t\big](V_{t+1}) \big) + o_p(N^{-1/2}).
\end{aligned}
\end{align*}
By assumption, the inner term $\big[\hat{\mathcal{Q}}_{t,N} - \mathcal{Q}_t\big](\hat{V}_{t+1,N}) - \big[\hat{\mathcal{Q}}_{t,N} - \mathcal{Q}_t\big](V_{t+1})$ is $o_p(N^{-1/2})$
in $C(\mathcal{X}_t \times \mathcal{U}_t) $. Because the linear operator $G_t'(\mathcal{Q}_t(V_{t+1}); \cdot)$ is bounded, it preserves this rate, yielding the desired result.
\end{proof}

\section{Asymptotic expansions for SAA value functions and optimal values}

We now derive a first-order asymptotic expansion for the SAA value
functions and the optimal value.
Define $\widetilde{\mathcal{T}}_{t,N} \colon C(\mathcal{X}_{t+1}) \to C(\mathcal{X}_t)$ by
\begin{align*}
[\widetilde{\mathcal{T}}_{t,N} V](x_t)
\coloneqq 
\frac{1}{N}
\sum_{i=1}^{N}
\big[
f_t(x_t,\pi_t^*(x_t),\xi_{ti})
+
V\big(
F_t(x_t,\pi_t^*(x_t),\xi_{ti})
\big)
\big].
\end{align*}
For $t=T,\ldots,1$, define recursively
$
\widetilde{V}_{T+1,N}\coloneqq f_{T+1}
$
and
$\widetilde{V}_{t,N} = \widetilde{\mathcal{T}}_{t,N} \widetilde{V}_{t+1,N}$.

\begin{lemma}
\label{lem:asymptotic-expansion-saa-values}
Let \Cref{ass-1,ass-1',ass-2} hold, and suppose that
\eqref{eq:q-equicontinuity} is satisfied for all
$t \in \{1,\ldots,T-1\}$. Then, for
$t = 1, \ldots, T$,
$
\|\hat{V}_{t,N} - \widetilde{V}_{t,N}\|_\infty = o_p(N^{-1/2})
$.
In particular, for all $x_1 \in \mathcal{X}_1$,
\begin{align}
\label{eq:asymptotic-expansion-optimal-value}
\hat{V}_{1,N}(x_1)
=
\mathbb{E}_{\hat{P}_{1,N}\times\cdots\times\hat{P}_{T,N}}
\big[
\tsum_{t=1}^{T}
f_t(\boldsymbol{x}_t,\pi_t^*(\boldsymbol{x}_t),\xi_t)
+
f_{T+1}(\boldsymbol{x}_{T+1})
\big]
+
o_p(N^{-1/2}).
\end{align}
\end{lemma}

\begin{proof}
Throughout the proof,
we use the infimum operator 
$G_t$  from \Cref{lem:Gt-Hadamard}.
By \Cref{ass-2}(iv), $G_t$ is Hadamard differentiable at
$\mathcal{Q}_t(V_{t+1})$, with derivative
$
[G_t'(\mathcal{Q}_t(V_{t+1});\eta)](x_t)
=
\eta(x_t,\pi_t^*(x_t))
$.
Now, we proceed by backward induction on $t$.
For $t=T$, we have
$
\hat{V}_{T,N}
=
G_T\big(\hat{\mathcal{Q}}_{T,N}(f_{T+1})\big)
$,
and
$
V_T
=
G_T\big(\mathcal{Q}_T(f_{T+1})\big)
$.
Moreover, by the functional CLT used in the proof of \Cref{pr-1},
$
N^{1/2}\big[
\hat{\mathcal{Q}}_{T,N}(f_{T+1}) - \mathcal{Q}_T(f_{T+1})
\big]
\dst \mathfrak{Y}_T
$.
Hence, the Delta Theorem yields, in $C(\mathcal{X}_t)$, the asymptotic expansion
\[
\hat{V}_{T,N}
=
V_T
+
\big[
\hat{\mathcal{Q}}_{T,N}(f_{T+1}) - \mathcal{Q}_T(f_{T+1})
\big](\cdot,\pi_T^*(\cdot))
+
o_p(N^{-1/2}).
\]
Since
$
V_T(\cdot)
=
\mathcal{Q}_T(f_{T+1})(\cdot,\pi_T^*(\cdot))
$,
we obtain in $C(\mathcal{X}_t)$,
\[
\hat{V}_{T,N}
=
\widetilde{\mathcal{T}}_{T,N}f_{T+1}
+
o_p(N^{-1/2})
=
\widetilde{V}_{T,N}
+
o_p(N^{-1/2}),
\]
and therefore
$
\|\hat{V}_{T,N} - \widetilde{V}_{T,N}\|_\infty
=
o_p(N^{-1/2})
$.

Now fix $t \in \{1,\ldots,T-1\}$, and assume as induction hypothesis
that
$
\|\hat{V}_{t+1,N} - \widetilde{V}_{t+1,N}\|_\infty
=
o_p(N^{-1/2})
$.
We have
$
\hat{V}_{t,N}
=
G_t\big(\hat{\mathcal{Q}}_{t,N}(\hat{V}_{t+1,N})\big)
$
and
$
V_t
=
G_t\big(\mathcal{Q}_t(V_{t+1})\big)
$.

Since \eqref{eq:q-equicontinuity} holds, \Cref{thm:nonunique-limits}
implies that 
$
N^{1/2}\big[
\hat{\mathcal{Q}}_{t,N}(\hat{V}_{t+1,N}) - \mathcal{Q}_t(V_{t+1})
\big]
\dst \mathfrak{Z}_t
$.
Therefore, applying the Delta Theorem to $G_t$ at
$\mathcal{Q}_t(V_{t+1})$, we obtain in $C(\mathcal{X}_t)$,
\[
\hat{V}_{t,N}
=
V_t
+
\big[
\hat{\mathcal{Q}}_{t,N}(\hat{V}_{t+1,N}) - \mathcal{Q}_t(V_{t+1})
\big](\cdot,\pi_t^*(\cdot))
+
o_p(N^{-1/2}).
\]
Since
$
V_t(\cdot)
=
\mathcal{Q}_t(V_{t+1})(\cdot,\pi_t^*(\cdot))
$,
it follows that in $C(\mathcal{X}_t)$,
\[
\hat{V}_{t,N}
=
\widetilde{\mathcal{T}}_{t,N}\hat{V}_{t+1,N}
+
o_p(N^{-1/2}).
\]
Since $\widetilde{\mathcal{T}}_{t,N}$ is Lipschitz continuous
with Lipschitz  constant one, the induction hypothesis and
the above asymptotic expansion ensure
\[
\begin{aligned}
\|\hat{V}_{t,N} - \widetilde{V}_{t,N}\|_\infty
&\le
\|\hat{V}_{t,N} -
\widetilde{\mathcal{T}}_{t,N}\hat{V}_{t+1,N}\|_\infty
 +
\|\widetilde{\mathcal{T}}_{t,N}\hat{V}_{t+1,N}
-
\widetilde{\mathcal{T}}_{t,N}\widetilde{V}_{t+1,N}\|_\infty
\\
&\le
\|\hat{V}_{t,N} -
\widetilde{\mathcal{T}}_{t,N}\hat{V}_{t+1,N}\|_\infty
 +
\|\hat{V}_{t+1,N}
-
\widetilde{V}_{t+1,N}\|_\infty
=
o_p(N^{-1/2}),
\end{aligned}
\]
This proves
$
\|\hat{V}_{t,N} - \widetilde{V}_{t,N}\|_\infty
=
o_p(N^{-1/2})
$
for $t=1,\ldots,T$.

Finally, the asymptotic expansion in \eqref{eq:asymptotic-expansion-optimal-value} 
follows from the dynamic programming principle.
\end{proof}

\begin{remark}
\label{rem-variance}
\normalfont
The asymptotic expansion in
\eqref{eq:asymptotic-expansion-optimal-value}
is the counterpart of the corresponding formula in the one-stage
setting; see \cite[eq.\ (5.24)]{SDR}. We note that the expectation
in \eqref{eq:asymptotic-expansion-optimal-value} is taken with respect to
the product empirical measure. The expectation $\mathbb{E}_{\hat{P}_{1,N}\times\cdots\times\hat{P}_{T,N}}$ is an empirical counterpart of the respective expectation $\mathbb{E}_{P_1\times\cdots\times P_T}$. In both cases such expectation is performed iteratively  going backward in time. That is, for a measurable function
$Z \colon \Xi_1 \times \cdots \times \Xi_T \to \mathbb{R}$, the conditional expectation $\bbe_{\hat{P}_{T,N}}[Z|\xi_{[T-1]}]$ is $ \tfrac{1}{N}
\sum_{i=1}^N
Z(\xi_{[T-1]},\xi_{Ti})$.
By continuing this process going backward in time, we obtain 
\begin{equation}\label{empiricalexpect}
\mathbb{E}_{\hat{P}_{1,N}\times\cdots\times\hat{P}_{T,N}}
\big[
Z(\xi_1,\ldots,\xi_T)
\big]
=
\frac{1}{N^T}
\sum_{i_1,\ldots,i_T=1}^N
Z(\xi_{1i_1},\ldots,\xi_{Ti_T}).
\end{equation}
Hence this quantity averages $Z$ over all $N^T$ cross-stage
combinations of the samples, rather than over the $N$ diagonal samples
$(\xi_{1i},\ldots,\xi_{Ti})$. Therefore, the asymptotic variance of
$\mathfrak{G}_1(x_1)$ need not coincide with the variance of the usual
static estimator based on the diagonal samples
$(\xi_{1i},\ldots,\xi_{Ti})$. \Cref{exam-contr} shows that the former can
be strictly smaller than the latter.
\hfill $\square$
\end{remark}

\bibliography{CLT-SAA-SOC.bbl}

\end{document}